\documentclass[11pt]{amsart}

\usepackage{amsmath, amsthm, amssymb, a4wide, bm}

\title[The Moduli Space of Polynomial Maps and Their Holomorphic Indices: I.%
]{The Moduli Space of Polynomial Maps and Their Holomorphic Indices: I. Generic Properties in the case of Having Multiple Fixed Points}
\author{Toshi Sugiyama}
\thanks{This work was partially supported by JSPS KAKENHI Grant Number JP19K14557}
\date{\today}

\address{Mathematics Studies, Gifu Pharmaceutical University, Mitahora-Higashi 5-6-1, Gifu-city, Gifu 502-8585, Japan}
\email{sugiyama-to@gifu-pu.ac.jp, sugiyama.toshi@gmail.com}

\subjclass[2010]{Primary 37F45; Secondary 15A99, 14C17
}
\keywords{complex dynamics, moduli space, polynomial maps, holomorphic index, fixed-point multiplier, multiple fixed point, linear algebra, B\'{e}zout's theorem, intersection multiplicity}

\theoremstyle{plain}
\newtheorem*{maintheorem}{Main Theorem}
\newtheorem*{mcorollary}{Corollary}

\newtheorem{theorem}{Theorem}[section]

\newtheorem{lemma}[theorem]{Lemma}
\newtheorem{proposition}[theorem]{Proposition}

\theoremstyle{definition}
\newtheorem{remark}[theorem]{Remark}
\newtheorem{definition}[theorem]{Definition}

\numberwithin{equation}{section}

\begin{document}

\begin{abstract}
 Following the author's previous works,
 we continue to consider the problem of counting the number of affine conjugacy classes of polynomials of one complex variable 
 when its unordered collection of holomorphic fixed point indices is given.
 The problem was already solved completely in the case that the polynomials have no multiple fixed points,
 in the author's previous papers.
 In this paper, we consider the case of having multiple fixed points, and obtain the formulae 
 for generic unordered collections of holomorphic fixed point indices,
 for each given degree and for each given number of fixed points.
\end{abstract}

\maketitle

\section{Introduction}

This paper is a continuation of the author's previous works~\cite{sugi1} and~\cite{sugi2}.

We first remind our setting from~\cite{sugi1} and~\cite{sugi2}.
Let $\mathrm{MP}_d$ be the family of affine conjugacy classes of polynomial maps of one complex variable with degree $d \geq 2$,
and $\mathbb{C}^d / \mathfrak{S}_d$ the set of unordered collections of $d$ complex numbers, where ${\mathfrak S}_{d}$ denotes the $d$-th symmetric group.
We denote by $\Phi_d$ the map 
\[
 \Phi_d : \mathrm{MP}_d \to \widetilde{\Lambda}_d \subset \mathbb{C}^d / \mathfrak{S}_d
\]
which maps  each $f \in \mathrm{MP}_d$ to its unordered collection of fixed-point multipliers.
Here, fixed-point multipliers of $f \in \mathrm{MP}_d$ always satisfy certain relation by the fixed point theorem for polynomial maps (see Proposition~\ref{prop1.2}), which implies that
the image of $\Phi_d$ is contained in a certain hyperplane $\widetilde{\Lambda}_d$ in $\mathbb{C}^d / \mathfrak{S}_d$.
We also denote by $\mathrm{MC}_d$ the family of monic centered polynomials of one complex variable with degree $d \geq 2$, 
and by 
\[
 \widehat{\Phi}_d: \mathrm{MC}_d \to \widetilde{\Lambda}_d
\]
the composite mapping of the natural projection $\mathrm{MC}_d \to \mathrm{MP}_d$
and $\Phi_d$.
As mentioned in~\cite{sugi1}, 
it is important 
to find the number of elements of each fiber of the maps $\Phi_d$ and $\widehat{\Phi}_d$, 
in the study of algebraic properties of moduli of polynomial maps. 

Motivated by some concerning 
results~\cite{Mc},~\cite{mi_qua},~\cite{tani},~\cite{HutzTepper},~\cite{Gorbovickis1},~\cite{Gorbovickis2}
and 
following some preliminary works 
such as~\cite{mi_cub},~\cite{NishizawaFujimura},~\cite{Fujimura2} and~\cite{fu},
we obtained, in~\cite{sugi1} and~\cite{sugi2}, explicit formulae for finding 
the number of elements of each fiber $\Phi_d^{-1}(\bar{\lambda})$ 
and also the number of elements of each fiber $\widehat{\Phi}_d^{-1}(\bar{\lambda})$
for every $\bar{\lambda} = \{\lambda_1, \dots, \lambda_d \} \in \widetilde{\Lambda}_d$
except when $\lambda_i = 1$ for some $i$.
In this paper we proceed to the next step, in which we 
consider the case where $\lambda_i = 1$ for some $i$.
However in this case the situation 
is much complicated, compared with the case where $\lambda_i \ne 1$ for every $i$.
In particular, if $\#\{ i \mid 1\leq i \leq d,\ \lambda_i=1 \} \geq 4$, then
the fibers $\Phi_d^{-1}(\bar{\lambda})$ and $\widehat{\Phi}_d^{-1}(\bar{\lambda})$
can have dimension greater than or equal to one.
To overcome this difficulty, we consider holomorphic fixed point indices in place of fixed-point multipliers,
%
 and
modify the maps $\Phi_d, \widehat{\Phi}_d$ so that the target space of the modified maps is
the set of unordered collections of holomorphic fixed point indices.
Here, holomorphic fixed point index is in some sense a similar one to fixed-point multiplier, 
but gives more detailed information if the fixed point is multiple (see Definition~\ref{def1.1}).
Under this modification, we obtain, in this paper, {\it generic} properties 
for the number of elements of a fiber of the modified maps.


\vspace{15pt}

In the rest of this section, we express the above mentioned explicitly and state the main theorem
in this paper.

We first fix our notation, some of which are the same as 
in~\cite{sugi1} and~\cite{sugi2}.
For $d \ge 2$, we put
\begin{align*}
 \mathrm{Poly}_d &:= \left\{f \in \mathbb{C}[z] \bigm| \deg f = d  \right\}, \\
 \mathrm{MC}_d &:= \left\{\left. f(z) = z^d + \sum_{k=0}^{d-2}a_kz^k \ \right| \ a_k \in \mathbb{C} \ \ \text{for} \ \ 0 \leq k \leq d-2 \right\} \ \text{and}\\
 \mathrm{Aut}(\mathbb{C})
 &:= \left\{\gamma (z) = az+b \bigm| a,b \in \mathbb{C},\ a \ne 0 \right\}.
\end{align*}
Since $\gamma \in \mathrm{Aut}(\mathbb{C})$
naturally acts on $f \in \mathrm{Poly}_d$
by $\gamma \cdot f := \gamma \circ f \circ \gamma^{-1}$,
we can define its quotient 
\[
  \mathrm{MP}_d := \mathrm{Poly}_d / \mathrm{Aut}(\mathbb{C}),
\]
which we usually call the moduli space of polynomial maps
of degree $d$.
Here, the equivalence class of $f \in \mathrm{Poly}_d$ in $\mathrm{MP}_d$ is called 
the affine conjugacy class of $f$.
An element of $\mathrm{MC}_d$ is called a monic centered polynomial of degree $d$.
We clearly have $\mathrm{MC}_d \subset \mathrm{Poly}_d$, and denote by
\[
  p: \mathrm{MC}_d \to \mathrm{MP}_d
\]
the natural projection 
$\mathrm{MC}_d \subset \mathrm{Poly}_d \twoheadrightarrow 
 \mathrm{Poly}_d / \mathrm{Aut}(\mathbb{C}) = \mathrm{MP}_d$.
The map $p$ is surjective, and the action of the group
$\{ az \in \mathrm{Aut}(\mathbb{C}) \mid a \in \mathbb{C},  a^{d-1} = 1 \} \cong \mathbb{Z} /(d-1)\mathbb{Z}$
on $\mathrm{MC}_d$ induces the isomorphism
$\overline{p} : \mathrm{MC}_d / \left( \mathbb{Z} /(d-1)\mathbb{Z} \right) \cong \mathrm{MP}_d$.

We put 
\[
  \mathrm{Fix}(f) := \{ z \in \mathbb{C} \bigm| f(z)=z\}
\]
for $f \in \mathrm{Poly}_d$, 
where $\mathrm{Fix}(f)$ is {\it not} considered counted with multiplicity in this paper.

\begin{definition}\label{def1.1}
 For $f \in \mathrm{Poly}_d$ and $\zeta \in \mathrm{Fix}(f)$,
	\begin{enumerate}
	 \item the derivative $f'(\zeta)$ is called the multiplier of $f$ at a fixed point $\zeta$.
	 \item We put 
		\[
		  \iota\left( f, \zeta \right) := 
		  \frac{1}{2\pi \sqrt{-1}}\oint_{\left| z-\zeta \right|=\epsilon} \frac{dz}{z-f(z)}
		\]
		for sufficiently small $\epsilon>0$.
		The residue $\iota\left( f, \zeta \right)$ is called the holomorphic index of $f$ 
		at a fixed point $\zeta$.
	\end{enumerate}
\end{definition}
Note that a fixed point $\zeta \in \mathrm{Fix}(f)$ is multiple 
if and only if $f'(\zeta) = 1$.
Moreover if $\zeta \in \mathrm{Fix}(f)$ is not multiple, then we always have 
$\iota\left( f, \zeta \right) = \frac{1}{1 - f'(\zeta)}$ 
by residue theorem.
It is also well known that the holomorphic index $\iota\left( f, \zeta \right)$ is invariant 
under the change of holomorphic coordinates
even when $f'(\zeta) = 1$.
Hence holomorphic index is a very similar object to multiplier.
In particular, in the case $f'(\zeta) \ne 1$, these two give the equivalent information;
however if $f'(\zeta) = 1$, holomorphic index gives more detailed information than multiplier
 (see section~12 in~\cite{mi_book} for instance).

\begin{proposition}[Fixed Point Theorem]\label{prop1.2}
 For $f \in \mathrm{Poly}_d$ we have 
 \[
   \sum_{\zeta \in \mathrm{Fix}(f)} \iota\left( f, \zeta \right) = 0.
 \]
\end{proposition}
Proposition~\ref{prop1.2} can easily be seen by the integration 
$\frac{1}{2\pi\sqrt{-1}} \oint_{|z|=R} \frac{dz}{z- f(z)}$
for sufficiently large real number $R$.

\vspace{5pt}

Every $f(z) \in \mathrm{Poly}_d$ can be expressed in the form
\[
   f(z) = z + \rho (z - \zeta_1)^{d_1} \cdots  (z - \zeta_{\ell})^{d_{\ell}},
\]
where $d_1,\dots,d_{\ell}$ are positive integers with $d_1 + \dots + d_{\ell} = d$, 
$\rho$ is a non-zero complex number, and
$\zeta_1, \dots, \zeta_{\ell}$ are mutually distinct complex numbers.
In this expression, we have $\#\mathrm{Fix}(f) = \ell$ and $\mathrm{Fix}(f) = \left\{ \zeta_1, \dots, \zeta_{\ell} \right\}$.
For such $f(z)$, we put
\[
   \mathrm{mult}(f,\zeta_i) := d_i 
\]
for $1 \leq i \leq \ell$, which we usually call the fixed-point multiplicity of $f$ at $\zeta_i$.

\begin{definition}
 For $d\geq 2$, we put
	\[
	  {\rm Mult}_d := \left\{ (d_1,\dots,d_\ell) \ \left| \ 
		\begin{matrix}
		\ell \geq 1,\quad d_1,\dots,d_\ell \in \mathbb{N}, \\
		d_1 + \dots + d_{\ell} = d,\\
		1 \leq d_1 \leq \dots \leq d_\ell
		\end{matrix} \right. \right\}.
	\]
 Moreover for each $(d_1,\dots,d_\ell) \in {\rm Mult}_d$, we put
	\begin{align*}
 	 {\rm Poly}_d(d_1,\dots,d_\ell) 
		&:= \left\{
		z + \rho(z - \zeta_1)^{d_1}\cdots (z - \zeta_\ell)^{d_\ell} \ \left| \ \begin{matrix}
		\rho, \zeta_1, \dots, \zeta_{\ell} \in \mathbb{C},\quad \rho \ne 0, \\
		\zeta_1,\dots,\zeta_{\ell} \text{ are mutually distinct}
		\end{matrix}\right.\right\},\\
	 {\rm MP}_d(d_1,\dots,d_\ell) 
		&:= {\rm Poly}_d(d_1,\dots,d_\ell) / \mathrm{Aut}(\mathbb{C}),\\
	 {\rm MC}_d(d_1,\dots,d_\ell) 
		&:= {\rm Poly}_d(d_1,\dots,d_\ell) \cap {\rm MC}_d \qquad \text{and}\\
	\Lambda_d(d_1,\dots,d_\ell) 
		&:= \left\{ 
		\left\{ \left(d_1,m_1\right), \dots, \left(d_\ell,m_\ell\right) \right\} 
		\ \left| \ \begin{matrix}
			m_1,\dots, m_\ell \in \mathbb{C}, \\
			m_1 + \dots + m_\ell = 0
		\end{matrix}\right.\right\},
	\end{align*}
 where $\left\{ \left(d_1,m_1\right), \dots, \left(d_\ell,m_\ell\right) \right\} $
 denotes the unordered collection of the pairs 
 $\left(d_1,m_1\right), \dots$, $\left(d_\ell,m_\ell\right)$.
\end{definition}
We naturally have the following stratifications indexed by the set of fixed-point multiplicities ${\rm Mult}_d$:
 \begin{align*}
  {\rm Poly}_d &=  \coprod_{(d_1,\dots,d_\ell) \in {\rm Mult}_d} {\rm Poly}_d(d_1,\dots,d_\ell),\\
  {\rm MP}_d &=  \coprod_{(d_1,\dots,d_\ell) \in {\rm Mult}_d} {\rm MP}_d(d_1,\dots,d_\ell) \qquad \text{and}\\
  {\rm MC}_d &=  \coprod_{(d_1,\dots,d_\ell) \in {\rm Mult}_d} {\rm MC}_d(d_1,\dots,d_\ell),
 \end{align*}
where $\coprod$ denotes the disjoint union.
On each strata ${\rm MP}_d(d_1,\dots,d_\ell)$ and ${\rm MC}_d(d_1,\dots,d_\ell)$, 
we naturally have the maps
 \begin{align*}
    \Phi_d(d_1,\dots,d_\ell) :\ &{\rm MP}_d(d_1,\dots,d_\ell) \to \Lambda_d(d_1,\dots,d_\ell) \qquad \text{and}\\
    \widehat{\Phi}_d(d_1,\dots,d_\ell) :\ &{\rm MC}_d(d_1,\dots,d_\ell) \to \Lambda_d(d_1,\dots,d_\ell)
 \end{align*}
 by $f \mapsto \left\{ \left( {\rm mult}(f,\zeta), \iota(f,\zeta) \right) \mid
 \zeta \in {\rm Fix}(f) \right\}$, 
by Proposition~\ref{prop1.2}.

\begin{remark}
 For every $(d_1,\dots,d_\ell) \in {\rm Mult}_d$, we have
 \[
   \dim_{\mathbb{C}}{\rm MP}_d(d_1,\dots,d_\ell) = \dim_{\mathbb{C}}{\rm MC}_d(d_1,\dots,d_\ell)
   = \dim_{\mathbb{C}}\Lambda_d(d_1,\dots,d_\ell) = \ell - 1.
 \]
\end{remark}

\begin{remark}
 The maps $\Phi_d(1,\dots,1)
$ and
 $\widehat{\Phi}_d(1,\dots,1)
$
 are essentially the same as the maps $\Phi_d : \mathrm{MP}_d \to \widetilde{\Lambda}_d$
 and $\widehat{\Phi}_d: \mathrm{MC}_d \to \widetilde{\Lambda}_d$
 which were mainly considered in~\cite{sugi1} and~\cite{sugi2}
 if restricted on 
 $\widetilde{V}_d := \left\{\left. \bar{\lambda} = \{\lambda_1, \dots, \lambda_d \} \in \widetilde{\Lambda}_d \ \right|
 \lambda_i \ne 1 \text{ for every } i \right\}$,
 because they are the same under
 the correspondence $\widetilde{V}_d \ni \left\{ \lambda_1,\dots,\lambda_d \right\} \mapsto 
  \left\{ \left( 1, \frac{1}{1-\lambda_1} \right), \dots, \left( 1, \frac{1}{1-\lambda_d} \right) \right\} \in \Lambda_d(1,\dots,1)$.
 Hence in~\cite{sugi1} and~\cite{sugi2}, we already have the formulae for finding
 $\# \Phi_d(1,\dots,1)^{-1}(\overline{m})$ and $\# \widehat{\Phi}_d(1,\dots,1)^{-1}(\overline{m})$
 for every $\overline{m} \in \Lambda_d(1,\dots,1)$.
 In this paper we consider\\
 $\# \Phi_d(d_1,\dots,d_\ell)^{-1}(\overline{m})$ and $\# \widehat{\Phi}_d(d_1,\dots,d_\ell)^{-1}(\overline{m})$
 for every $(d_1,\dots,d_\ell) \in {\rm Mult}_d$ and for generic $\overline{m} \in \Lambda_d(d_1,\dots,d_\ell)$.
\end{remark}

We now state the main theorem in this paper.

\begin{maintheorem}
 Let $(d_1,\dots,d_\ell)$ be an element of ${\rm Mult}_d$ with $\ell \geq 2$. Then
 \begin{enumerate}
  \item for every $\overline{m} = \left\{ \left(d_1,m_1\right), \dots, \left(d_\ell,m_\ell\right) \right\}
	\in \Lambda_d(d_1,\dots,d_\ell)$, we have 
	\[
	   \# \Phi_d(d_1,\dots,d_\ell)^{-1}(\overline{m}) \leq \dfrac{(d-2)!}{(d-\ell)!} \quad \text{and} \quad
	   \# \widehat{\Phi}_d(d_1,\dots,d_\ell)^{-1}(\overline{m}) \leq \dfrac{(d-1)!}{(d-\ell)!}.
	\]
  \item For $\overline{m} = \left\{ \left(d_1,m_1\right), \dots, \left(d_\ell,m_\ell\right) \right\}
	\in \Lambda_d(d_1,\dots,d_\ell)$, 
	the implications $(c) \Leftrightarrow (b) \Rightarrow (a)$ always hold
	for the following three conditions:
	\begin{enumerate}
	 \item $\displaystyle \# \Phi_d(d_1,\dots,d_\ell)^{-1}(\overline{m}) = \dfrac{(d-2)!}{(d-\ell)!}$
	 \item $\displaystyle \# \widehat{\Phi}_d(d_1,\dots,d_\ell)^{-1}(\overline{m}) = \dfrac{(d-1)!}{(d-\ell)!}$
	 \item $\left(d_1,m_1\right), \dots, \left(d_\ell,m_\ell\right)$ are mutually distinct, and 
		$\sum_{i \in I}m_i \ne 0$ holds for every $\emptyset \ne I \subsetneq \{1,\dots,\ell\}$.
		\label{conditionM2c}
	\end{enumerate}
	Moreover the implication $(a) \Rightarrow (c)$ also holds except in the case $d = \ell = 3$.
 \end{enumerate}
\end{maintheorem}

\begin{remark}
 As mentioned in the first page in~\cite{sugi1}, the map $\Phi_3 : \mathrm{MP}_3 \to \widetilde{\Lambda}_3$ is bijective.
 Hence in the case $d=\ell = 3$, the equality $\# \Phi_3(1,1,1)^{-1}(\overline{m}) = 1$
 holds for 
 $\overline{m} = \left\{ \left(1,m_1\right), \left(1,m_2\right), \left(1,m_3\right) \right\}
 \in \Lambda_3(1,1,1)$
 if and only if $m_i \ne 0$ holds for every $1 \leq i \leq 3$.
 Hence in this case, the implication $(a) \Rightarrow (c)$ does not hold
 since $\left(1,m_1\right), \left(1,m_2\right), \left(1,m_3\right)$ are not always mutually distinct.
\end{remark}

The set of $\overline{m} \in \Lambda_d(d_1,\dots,d_\ell)$ satisfying
the condition~(\ref{conditionM2c}) in Main Theorem is Zariski open in $\Lambda_d(d_1,\dots,d_\ell)$; 
hence we have the following:

\begin{mcorollary}
 Let $(d_1,\dots,d_\ell)$ be an element of ${\rm Mult}_d$ with $\ell \geq 2$. Then
 \begin{enumerate}
  \item $\Phi_d(d_1,\dots,d_\ell)$ is generically a $\frac{(d-2)!}{(d-\ell)!}$-to-one map.
  \item $\widehat{\Phi}_d(d_1,\dots,d_\ell)$ is generically a $\frac{(d-1)!}{(d-\ell)!}$-to-one map.
 \end{enumerate}
\end{mcorollary}

\begin{remark}
 In the case $\ell = 1$, we have $(d_1,\dots,d_\ell) = (d)$ and also have
 $\# {\rm MP}_d(d) = \# {\rm MC}_d(d) = \# \Lambda_d(d) = 1$.
 Hence in this case, the maps $\Phi_d(d)$ and $\widehat{\Phi}_d(d)$ are trivially bijective.
\end{remark}


We have three sections in this paper.
The most frequently used tool for the proof of Main Theorem is linear algebra, especially Proposition~\ref{prop2.4}.
Section~\ref{section2} is devoted to introduce some results in linear algebra including Proposition~\ref{prop2.4}.
On the other hand, 
the proof of  Main Theorem itself is given in Section~\ref{section3}.
Most steps in the proof of Main Theorem 
are analogies
of the proofs of the main theorems in~\cite{sugi1}; 
however in almost all the steps, its proof is much complicated, compared with the original one in~\cite{sugi1}.
Moreover in~\cite{sugi1}, there does not exist a counterpart for Proposition~\ref{prop3.5},
which is the most crucial part in the proof of Main Theorem from the standpoint of technique.


\section{preparation from linear algebra}\label{section2}


In this section, we remind our notations and propositions in linear algebra
which were given in the latter half of Section~7 in~\cite{sugi1}.
These are also used very often in this paper throughout Section~\ref{section3} in the proof of Main Theorem.

\begin{definition}
 For non-negative integers $n,b,k,h$ with $n>k$ and $b>h$,
 we denote by $A_{n,k}^{b,h}(\alpha)$ the $(n-k,b-h)$ matrix
 whose $(i,j)$-th entry is $\binom{i+k-1}{j+h-1} \alpha^{(i+k)-(j+h)}$
 for $1 \leq i \leq n-k$ and $1 \leq j \leq b-h$.
 Moreover we put $A_{n,k}^b(\alpha):=A_{n,k}^{b,0}(\alpha)$ and
 $A_n^b(\alpha):=A_{n,0}^{b,0}(\alpha)$.
\end{definition}

By definition,
the matrix $A_{n,k}^{b,h}(\alpha)$ is obtained from the $(n,b)$ matrix
\[
 A_n^b(\alpha)
  = \begin{pmatrix}
      1        & 0         & 0         & 0       & \cdots & 0      \\
      \alpha   & 1         & 0         & 0       & \cdots & 0      \\
      \alpha^2 & 2\alpha   & 1         & 0       & \cdots & 0      \\
      \alpha^3 & 3\alpha^2 & 3\alpha   & 1       & \cdots & 0      \\
      \vdots   & \vdots    & \vdots    & \vdots  & \ddots & \vdots \\
      \alpha^{n-1} & (n-1)\alpha^{n-2} & \binom{n-1}{2}\alpha^{n-3} &
        \binom{n-1}{3}\alpha^{n-4} & \cdots & 
     \end{pmatrix}
\]
by cutting off the upper $k$ rows and the left $h$ columns.

\begin{definition}
 For a positive integer $b$,
 we denote by $X_b$ the $(b,b)$ diagonal matrix
 whose $(i,i)$-th entry is $i$ for $1 \le i \le b$.
 Moreover we denote by $I_b$ the $(b,b)$ identity matrix,
 and by $N_b$ the $(b,b)$ nilpotent matrix
 whose $(i,i+1)$-st entry is $1$ for $1 \le i \le b-1$
 and whose other entries are $0$,
 i.e.,
 \[
  X_b = \begin{pmatrix}
	 1      & 0      & \cdots & 0 \\
	 0      & 2      & \cdots & 0 \\
	 \vdots & \vdots & \ddots & \vdots \\
	 0      & 0      & \cdots & b
	\end{pmatrix}, \quad
  I_b=  \begin{pmatrix}
	 1 & 0 & \cdots & 0 \\
	 0 & 1 & \cdots & 0 \\
	 \vdots & \vdots & \ddots & \vdots \\
	 0 & 0 & \cdots & 1
	\end{pmatrix}
  \quad \text{and} \quad
  N_b = \begin{pmatrix}
	 0      & 1      & 0      & \cdots & 0 \\
	 0      & 0      & 1      & \cdots & 0 \\
	 \vdots & \vdots & \vdots & \ddots & \vdots \\
	 0      & 0      & 0      & \cdots & 1 \\
	 0      & 0      & 0      & \cdots & 0
	\end{pmatrix}.
 \]
\end{definition}

\begin{proposition}\label{prop2.3}
 For positive integers $n$ and $b$, we have $A_{n+1,1}^{b+1,1}(\alpha) = X_n \cdot A_n^b(\alpha) \cdot {X_b}^{-1}$.
\end{proposition}
\begin{proof}
 This can easily be verified by $\binom{i}{j}=\binom{i-1}{j-1}\cdot\frac{i}{j}$.
\end{proof}

\begin{proposition}\label{prop2.4}
 Let $r_1, \dots, r_{\ell}, r$ be positive integers with $r = r_1 + \dots + r_{\ell}$.
 Then we have 
 \begin{align*}
     \det \bigl( A_r^{r_1}(\alpha_1), \dots, A_r^{r_\ell}(\alpha_\ell) \bigr) 
	&= \prod_{1 \leq v < u \leq \ell} \left( \alpha_u - \alpha_v \right)^{r_v r_u} \qquad \text{and}\\
     \det \left( A_{r+1, 1}^{r_1+1,1}(\alpha_1), \dots, A_{r+1,1}^{r_\ell+1,1}(\alpha_\ell) \right) 
	&= \frac{r!}{r_1!\cdots r_\ell!}\cdot
	     \prod_{1 \leq v < u \leq \ell} \left( \alpha_u - \alpha_v \right)^{r_v r_u}.
 \end{align*}
\end{proposition}
\begin{proof}
 See the proof of Lemma~7.8 in~\cite{sugi1}.
 Note that the latter equality is a direct consequence of the former one and Proposition~\ref{prop2.3}.
\end{proof}

\section{proof}\label{section3}

In this section, we prove Main Theorem by using the propositions in Section~\ref{section2}.
We always assume the following throughout this section:
\begin{itemize}
 \item $\ell$ and $d$ are integers greater than or equal to $2$.
 \item $(d_1,\dots,d_\ell)$ is an element of ${\rm Mult}_d$.
 \item $\rho, \zeta_1, \dots, \zeta_{\ell}$ are complex numbers with $\rho \ne 0$.
 \item $f(z) = z + \rho \left( z - \zeta_1 \right)^{d_1} \cdots \left( z - \zeta_{\ell} \right)^{d_{\ell}}$.
 	Hence $f(z) \in {\rm Poly}_d(d_1,\dots,d_{\ell})$ holds if and only if $\zeta_1, \dots, \zeta_{\ell}$ are mutually distinct.
 \item $m= (m_1,\dots, m_{\ell})$ is an element of $\mathbb{C}^{\ell}$ with $m_1 + \dots + m_{\ell} = 0$.
 	Moreover for such $m$, we put $\overline{m} := \left\{ \left(d_1,m_1\right), \dots, \left(d_\ell,m_\ell\right) \right\}$.
 	Hence we always have $\overline{m} \in \Lambda_d(d_1,\dots,d_{\ell})$.
 \item $m$ is assumed not to be equal to $(0,\dots,0)$, except in Propositions~\ref{prop3.1},~\ref{prop3.2} and~\ref{prop3.3}.
\end{itemize}

We first prove the following proposition, 
which is the first step
for the proof of Main Theorem, and 
is also an analogue of Key Lemma in Section 4 in~\cite{sugi1} 
in the case of having multiple fixed points.
However Key Lemma in Section 4 in~\cite{sugi1} was much simpler than the following proposition.

\begin{proposition}\label{prop3.1}
 Suppose that $\zeta_1, \dots, \zeta_{\ell}$ are mutually distinct. Then
 the following two conditions~{\rm (\ref{condition3.1})} and~{\rm (\ref{condition3.2})} are equivalent:
 \begin{enumerate}
  \item The equalities $\iota\left( f,\zeta_i \right) = m_{i}$ hold for $1 \leq i \leq \ell$.
	 \label{condition3.1}
  \item There exist $m_{i,k} \in \mathbb{C}$ for $1 \leq i \leq \ell$ and $1 \leq k\leq d_i-1$ 
	such that the equality
	\begin{equation}\label{eq3.1}
	   \sum_{i=1}^{\ell} A_{d}^{d_i}\left(\zeta_i\right) 
	   \begin{pmatrix} m_{i} \\ m_{i,1}\\ \vdots \\ m_{i,d_i-1} \end{pmatrix} 
	   = \begin{pmatrix} 0 \\ \vdots \\ 0 \\ -\frac{1}{\rho} \end{pmatrix}
	\end{equation}
	holds,
	where in the case $d_i = 1$, the column vector $^t\! \left( m_i, m_{i,1},\dots, m_{i,d_i-1} \right)$ is assumed to be $(m_i)$.
	 \label{condition3.2}
 \end{enumerate}
\end{proposition}

\begin{proof}
 Since $\deg\left( z - f(z) \right) = d \geq 2$, the equalities
 \[
    \frac{1}{2\pi\sqrt{-1}} \oint_{|z|=R} \frac{z^k}{z- f(z)}dz = 
     \begin{cases}
	0 			 & (k=0,1,\dots,d-2) \\
	-\frac{1}{\rho} &(k=d-1)
     \end{cases}
 \]
 hold for sufficiently large real number $R$.
 On the other hand, by the residue theorem, we have
 \begin{align*}
   \frac{1}{2\pi\sqrt{-1}} \oint_{|z|=R} \frac{z^k}{z- f(z)}dz
    &= \sum_{i=1}^{\ell} \frac{1}{2\pi\sqrt{-1}} \oint_{|z- \zeta_i|=\epsilon} 
	\frac{\left\{ (z - \zeta_i) + \zeta_i \right\}^k}{z- f(z)}dz\\
    &= \sum_{i=1}^{\ell} \sum_{h=0}^{\min\{k, d_i-1\}} \binom{k}{h}\!\cdot\!\zeta_i^{k-h}
	\!\cdot\! \frac{1}{2\pi\sqrt{-1}} \oint_{|z- \zeta_i|=\epsilon} 
	\frac{\left(z - \zeta_i\right)^h}{z- f(z)}dz
 \end{align*}
 for sufficiently small positive real number $\epsilon$.
 Hence putting 
 \[
    \iota_h(f, \zeta_i) := \frac{1}{2\pi\sqrt{-1}} \oint_{|z- \zeta_i|=\epsilon} 
	\frac{\left(z - \zeta_i\right)^h}{z- f(z)}dz
 \]
 for $h \geq 0$ and for sufficiently small $\epsilon >0$, we have $\iota_0(f, \zeta_i) = \iota(f, \zeta_i)$, 
 $\iota_h(f, \zeta_i) = 0$ for $h \geq d_i$, and also have
 \begin{equation}\label{eq3.3}
  \sum_{i=1}^{\ell} \sum_{h=0}^{d_i-1} \binom{k}{h}\!\cdot\!\zeta_i^{k-h}
	\!\cdot\! \iota_h(f, \zeta_i) = 
     \begin{cases}
	0 			 & (k=0,1,\dots,d-2) \\
	-\frac{1}{\rho} &(k=d-1)
     \end{cases}
 \end{equation}
 for every $f(z)=z+\rho \left( z - \zeta_1 \right)^{d_1} \cdots \left( z - \zeta_{\ell} \right)^{d_{\ell}} 
 \in {\rm Poly}_d(d_1,\dots,d_\ell)$.
 Moreover by using matrix, we find that the equalities~(\ref{eq3.3}) are equivalent to the equality
 \begin{equation}\label{eq3.4}
  \sum_{i=1}^{\ell} A^{d_i}_d(\zeta_i) 
    \begin{pmatrix} \iota_0(f, \zeta_i) \\ \iota_1(f, \zeta_i) \\ \vdots \\ \iota_{d_i-1}(f, \zeta_i)
    \end{pmatrix} 
  = \begin{pmatrix} 0 \\ \vdots \\ 0 \\ -\frac{1}{\rho} \end{pmatrix},
 \end{equation}
 which is also equivalent to
 \begin{equation}\label{eq3.5}
  \left( A^{d_1}_d(\zeta_1), A^{d_2}_d(\zeta_2), \dots, A^{d_{\ell}}_d(\zeta_{\ell}) \right) 
    \begin{pmatrix} 
      \iota_0(f, \zeta_1) \\ \vdots \\ \iota_{d_1-1}(f, \zeta_1)  \\ \vdots \\
      \iota_0(f, \zeta_{\ell}) \\ \vdots \\ \iota_{d_{\ell}-1}(f, \zeta_{\ell})
    \end{pmatrix}
  = \begin{pmatrix} 0 \\ \vdots \\ 0 \\ -\frac{1}{\rho} \end{pmatrix}.
 \end{equation}

 Hence 
 the condition~(1) implies the equality~(\ref{eq3.1})
 by putting $m_{i,k} = \iota_k(f, \zeta_i)$ for $1 \leq i \leq \ell$ and $1 \leq k\leq d_i-1$,
 which verifies the implication $(1) \Rightarrow (2)$.

 On the other hand, since the square matrix 
 $\left( A^{d_1}_d(\zeta_1), A^{d_2}_d(\zeta_2), \dots, A^{d_{\ell}}_d(\zeta_{\ell}) \right)$
 is invertible by Proposition~\ref{prop2.4},
 the equalities~(\ref{eq3.1}) and~(\ref{eq3.4}) (or~(\ref{eq3.5})) imply the equalities
 \[
   ^t\!\left( m_i, m_{i,1}, \dots, m_{i,d_i-1} \right) = \,\!
   ^t\!\left( \iota(f, \zeta_i), \iota_1(f, \zeta_i), \dots, \iota_{d_i-1}(f, \zeta_i) \right)
 \]
 for $1 \leq i \leq \ell$,
 which verifies the implication $(2) \Rightarrow (1)$.
\end{proof}

\begin{remark}
 In the rest of this section, we often use expressions like the equality~(\ref{eq3.4})
 in place of~(\ref{eq3.5}) for the simplicity of description as in the proof of Proposition~\ref{prop3.1}
\end{remark}

Concerning Proposition~\ref{prop3.1}, 
the following Propositions~\ref{prop3.2},~\ref{prop3.3} and~\ref{prop3.4} also hold.

\begin{proposition}\label{prop3.2}
 Suppose that $\zeta_1, \dots, \zeta_{\ell}$ are mutually distinct, and that the equality~{\rm (\ref{eq3.1})} holds
 for $m_{i,k} \in \mathbb{C}$ with $1 \leq i \leq \ell$ and $1 \leq k\leq d_i-1$.
 Then we have the following for $1 \leq i \leq \ell$:
 \begin{enumerate}
  \item $m_{i,d_i-1} \ne 0$ if $d_i \geq 2$.
  \item $m_i \ne 0$ if $d_i =1$.
 \end{enumerate}
\end{proposition}

\begin{proof}
 Without loss of generality, we may assume that $i=1$.
 Suppose $d_1 \geq 2$ and $m_{1,d_1-1}=0$. Then the equality~(\ref{eq3.1}) is equivalent to the equality
 \[
   A_{d}^{d_1-1}\left(\zeta_1\right) 
   \begin{pmatrix} m_{1} \\ m_{1,1}\\ \vdots \\ m_{1,d_1-2} \end{pmatrix} 
   + \sum_{i=2}^{\ell} A_{d}^{d_i}\left(\zeta_i\right) 
   \begin{pmatrix} m_{i} \\ m_{i,1}\\ \vdots \\ m_{i,d_i-1} \end{pmatrix} 
   = \begin{pmatrix} 0 \\ \vdots \\ 0 \\ -\frac{1}{\rho} \end{pmatrix},
 \]
 which implies
 \[
   A_{d-1}^{d_1-1}\left(\zeta_1\right) 
   \begin{pmatrix} m_{1} \\ m_{1,1}\\ \vdots \\ m_{1,d_1-2} \end{pmatrix} 
   + \sum_{i=2}^{\ell} A_{d-1}^{d_i}\left(\zeta_i\right) 
   \begin{pmatrix} m_{i} \\ m_{i,1}\\ \vdots \\ m_{i,d_i-1} \end{pmatrix} 
   = \begin{pmatrix} 0 \\ \vdots \\ 0 \end{pmatrix}.
 \]
 Since the square matrix 
 $\left( A^{d_1-1}_{d-1}(\zeta_1), A^{d_2}_{d-1}(\zeta_2), \dots, A^{d_{\ell}}_{d-1}(\zeta_{\ell}) \right)$
 is invertible by Proposition~\ref{prop2.4},
 we have $^t\!\left( m_i, m_{i,1}, \dots, m_{i,d_i-1} \right) = \,\! ^t\!\left( 0, 0,\dots, 0 \right)$
 for every $1 \leq i \leq \ell$.
 Hence the left-hand side of the equality~(\ref{eq3.1}) is equal to zero, which contradicts ~(\ref{eq3.1}).
 We therefore have the contradiction, which implies the assertion~(1).

 Suppose $d_1=1$ and $m_1 =0$ next.
 Then the equality~(\ref{eq3.1}) is, in this case, equivalent to
 \[
   \sum_{i=2}^{\ell} A_{d}^{d_i}\left(\zeta_i\right) 
   \begin{pmatrix} m_{i} \\ m_{i,1}\\ \vdots \\ m_{i,d_i-1} \end{pmatrix} 
   = \begin{pmatrix} 0 \\ \vdots \\ 0 \\ -\frac{1}{\rho} \end{pmatrix}.
 \]
 Hence by a similar argument to the above, 
 the invertibility of the square matrix \\
 $\left( A^{d_2}_{d-1}(\zeta_2), \dots, A^{d_{\ell}}_{d-1}(\zeta_{\ell}) \right)$ leads to a contradiction,
 which implies the assertion~(2).
\end{proof}

\begin{proposition}\label{prop3.3}
 Suppose that $f(z)$ is an element of ${\rm Poly}_d(d_1,\dots,d_\ell)$.
 Then $\iota\left( f,\zeta \right)  \ne 0$ holds for some $\zeta \in {\rm Fix}(f)$.
 Hence in the case $m=(0,\dots,0)$, we have 
 $\Phi_d(d_1,\dots,d_\ell)^{-1}(\overline{m}) = \widehat{\Phi}_d(d_1,\dots,d_\ell)^{-1}(\overline{m}) = \emptyset$.
\end{proposition}

\begin{proof}
 By the proof of Proposition~\ref{prop3.1}, 
 we always have the equality~(\ref{eq3.4})
 for \\
 $f(z) \in {\rm Poly}_d(d_1,\dots,d_\ell)$.
 Suppose 
 $\iota(f,\zeta_i)=0$ 
 for every $1 \leq i \leq \ell$.
 Then the equality~(\ref{eq3.4}) is equivalent to the equality
 \[
   \sum_{i=1}^{\ell} A^{d_i, 1}_{d,1}(\zeta_i) 
    \begin{pmatrix} \iota_1(f, \zeta_i) \\ \vdots \\ \iota_{d_i-1}(f, \zeta_i) \end{pmatrix} 
   = \begin{pmatrix} 0 \\ \vdots \\ 0 \\ -\frac{1}{\rho} \end{pmatrix},
 \]
 which implies
 \[
   \sum_{i=1}^{\ell} A^{d_i, 1}_{d-\ell+1,1}(\zeta_i) 
    \begin{pmatrix} \iota_1(f, \zeta_i) \\ \vdots \\ \iota_{d_i-1}(f, \zeta_i) \end{pmatrix} 
   = \begin{pmatrix} 0 \\ \vdots \\ 0 \end{pmatrix}
 \]
 since $\ell \geq 2$.
 Since the square matrix 
 $\left( A^{d_1, 1}_{d-\ell+1,1}(\zeta_1), A^{d_2, 1}_{d-\ell+1,1}(\zeta_2), \dots, 
 A^{d_{\ell}, 1}_{d-\ell+1,1}(\zeta_{\ell})  \right)$ is invertible by Proposition~\ref{prop2.4},
 we have $^t\!\left( \iota_1(f, \zeta_i), \dots, \iota_{d_i-1}(f, \zeta_i) \right) =\!\, ^t\!(0, \dots,0)$
 for every $1 \leq i \leq \ell$.
 Hence the left-hand side of the equality~(\ref{eq3.4}) must be equal to zero,
 which contradicts the equality~(\ref{eq3.4}).
 We therefore have $\iota(f, \zeta_i) \ne 0$ for some $1 \leq i \leq \ell$,
 which completes the proof of Proposition~\ref{prop3.3}.
\end{proof}

In the rest of this section,
$m = \left( m_1,\dots,m_\ell \right)$ is always assumed not to be equal to $(0,\dots,0)$,
as mentioned in the opening paragraph of this section.

\begin{proposition}\label{prop3.4}
 Suppose that $\zeta_1,\dots,\zeta_{\ell}$ are mutually distinct.
 If the equality
 \begin{equation}\label{eq3.2}
  \sum_{i=1}^{\ell} A_{d-1}^{d_i}\left(\zeta_i\right) 
   \begin{pmatrix} m_{i} \\ m_{i,1}\\ \vdots \\ m_{i,d_i-1} \end{pmatrix} 
  = \begin{pmatrix} 0 \\ \vdots \\ 0 \end{pmatrix}
 \end{equation}
 holds 
 for $m_{i,k} \in \mathbb{C}$ with $1 \leq i \leq \ell$ and $1 \leq k\leq d_i-1$,
 then there exists a unique non-zero complex number $\rho$ such that 
 the equality~{\rm (\ref{eq3.1})} holds.
\end{proposition}

\begin{proof}
 Suppose that such $\rho$ does not exist.
 Then the left-hand side of the equality~(\ref{eq3.1}) must be zero.
 Since the square matrix 
 $\left( A^{d_1}_d(\zeta_1), A^{d_2}_d(\zeta_2), \dots, A^{d_{\ell}}_d(\zeta_{\ell}) \right)$
 is invertible by Proposition~\ref{prop2.4}, we have 
 $^t\!\left( m_i, m_{i,1}, \dots, m_{i,d_i-1} \right) =\!\, ^t\!(0, 0, \dots,0)$
 for every $1 \leq i \leq \ell$,
 which contradicts $(m_1,\dots,m_{\ell}) \ne (0, \dots, 0)$.
 Hence the contradiction assures the existence of $\rho$ satisfying the equality~(\ref{eq3.1}).
\end{proof}

\vspace{10pt}

We proceed to the next step, in which
we exclude $m_{i,k}$ for $1 \leq i \leq \ell$ and $1 \leq k \leq d_i - 1$
from the equality~(\ref{eq3.2}).
From the standpoint of technique, this step is the most crucial in the proof of Main Theorem.
As can be verified in~\cite{sugi1}, this step does not exist in the case of having no multiple fixed points.

\begin{proposition}\label{prop3.5}
 Suppose that $\zeta_1,\dots,\zeta_{\ell}$ are mutually distinct. Then
 the following two conditions~{\rm (1)} and~{\rm (2)} are equivalent:
  \begin{enumerate}
   \item There exist $m_{i,k} \in \mathbb{C}$ for $1 \leq i \leq \ell$ and $1 \leq k\leq d_i-1$ 
	such that the equality~{\rm (\ref{eq3.2})} holds.
   \item The equality
	\begin{equation}\label{eq3.6}
	    A_{\ell-2}^{d-2}(0) \left(\prod_{i=1}^{\ell}\left( -\zeta_i I_{d-2} +N_{d-2} \right)^{d_i - 1}\right)
	    \left(X_{d-2}\right)^{-1}
	   \begin{pmatrix}
	    \zeta_1&  \dots & \zeta_{\ell} \\
	    \zeta_1^2&  \dots & \zeta_{\ell}^2 \\
	    \vdots & \ddots & \vdots\\
	    \zeta_1^{d-2} & \dots & \zeta_{\ell}^{d-2}
	   \end{pmatrix}
	   \begin{pmatrix}
	    m_1 \\ \vdots \\ m_{\ell}
	   \end{pmatrix}
	   = 0
	\end{equation}
	holds.
  \end{enumerate}
\end{proposition}

In Proposition~\ref{prop3.5}, note that in the case $\ell = 2$, the equality~(\ref{eq3.6}) always holds
since it can be considered to be an equality in $0$-dimensional $\mathbb{C}$-vector space $\mathbb{C}^0$.

Proposition~\ref{prop3.5} is obtained from Lemma~\ref{lem3.5} below
by substituting $q=\ell'=\ell$, $d'_i = d_i - 1$, $\alpha_i = \zeta_i$ and $m'_i = m_i$ for $1 \leq i \leq \ell$.
Lemma~\ref{lem3.5} is also utilized later in the proofs of Proposition~\ref{prop3.12} and Lemma~\ref{lem3.19}.

\begin{lemma}\label{lem3.5}
 Let $\ell', q, d'_1, \dots, d'_q$ be non-negative integers with $q \geq 1,\ \ell' \geq 2$ and $d'_1+\dots+d'_q=d-\ell'$.
 Moreover let $\alpha_1,\dots,\alpha_q$ be mutually distinct complex numbers, 
 and $m'_1,\dots,m'_q$ complex numbers with $m'_1 + \dots + m'_q = 0$.
 Then the following two conditions~{\rm (1)} and~{\rm (2)} are equivalent:
  \begin{enumerate}
   \item There exist $m_{u,k} \in \mathbb{C}$ for $1 \leq u \leq q$ and $1 \leq k\leq d'_u$ 
	such that the equality 
	\begin{equation}\label{eq3.2.2}
	   \sum_{u=1}^q A_{d-1}^{d'_u+1}\left(\alpha_u\right)
		\begin{pmatrix} m'_{u} \\ m_{u,1}\\ \vdots \\ m_{u,d'_u} \end{pmatrix}
	   = 0
	\end{equation}
	holds.
   \item The equality
	\begin{equation}\label{eq3.6.2}
	    A_{\ell'-2}^{d-2}(0) \left(\prod_{u=1}^{q}\left( -\alpha_u I_{d-2} +N_{d-2} \right)^{d'_u}\right)
	    \left(X_{d-2}\right)^{-1}
	   \begin{pmatrix}
	    \alpha_1&  \dots & \alpha_{q} \\
	    \alpha_1^2&  \dots & \alpha_{q}^2 \\
	    \vdots & \ddots & \vdots\\
	    \alpha_1^{d-2} & \dots & \alpha_{q}^{d-2}
	   \end{pmatrix}
	   \begin{pmatrix}
	    m'_1 \\ \vdots \\ m'_{q}
	   \end{pmatrix}
	   = 0
	\end{equation}
	holds.
  \end{enumerate}
\end{lemma}

\begin{proof}
 Note first that the first row of the equality~(\ref{eq3.2.2}) is the same as $m'_1+\dots+m'_{q}=0$,
 which is assumed in Lemma~\ref{lem3.5}.
 Hence the equality~(\ref{eq3.2.2}) is equivalent to the equality
 \begin{equation}\label{eq3.7}
	   \sum_{u=1}^q A_{d-1, 1}^{d'_u+1}\left(\alpha_u\right)
		\begin{pmatrix} m'_{u} \\ m_{u,1}\\ \vdots \\ m_{u,d'_u} \end{pmatrix}
	   = 0.
 \end{equation}
 Moreover since
 \begin{align*}
  \sum_{u=1}^q A_{d-1, 1}^{d'_u+1}\left(\alpha_u\right)
	&\begin{pmatrix} m'_{u} \\ m_{u,1}\\ \vdots \\ m_{u,d'_u} \end{pmatrix}
  = \sum_{u=1}^q \left\{ 
	m'_u \begin{pmatrix} \alpha_u \\ \alpha_u^2 \\ \vdots \\ \alpha_u^{d-2} \end{pmatrix} +
	A_{d-1, 1}^{d'_u+1,1}\left(\alpha_u\right)
	\begin{pmatrix} m_{u,1}\\ \vdots \\ m_{u,d'_u} \end{pmatrix}
    \right\}\\
  &= \sum_{u=1}^{q} X_{d-2} A_{d-2}^{d'_u}\left(\alpha_u\right) \left(X_{d'_u}\right)^{-1}
	\begin{pmatrix} m_{u,1}\\ \vdots \\ m_{u,d'_u} \end{pmatrix}
	+ \begin{pmatrix}
	    \alpha_1&  \dots & \alpha_{q} \\
	    \alpha_1^2&  \dots & \alpha_{q}^2 \\
	    \vdots & \ddots & \vdots\\
	    \alpha_1^{d-2} & \dots & \alpha_{q}^{d-2}
	   \end{pmatrix}
	   \begin{pmatrix}
	    m'_1 \\ \vdots \\ m'_{q}
	   \end{pmatrix}
 \end{align*}
 by Proposition~\ref{prop2.3},
 the equality~(\ref{eq3.7}) is equivalent to
 \begin{equation}\label{eq3.8}
   \sum_{u=1}^{q} A_{d-2}^{d'_u}\left(\alpha_u\right) \left(X_{d'_u}\right)^{-1}
	\begin{pmatrix} m_{u,1}\\ \vdots \\ m_{u,d'_u} \end{pmatrix}
	+ (X_{d-2})^{-1} \begin{pmatrix}
	    \alpha_1&  \dots & \alpha_{q} \\
	    \vdots & \ddots & \vdots\\
	    \alpha_1^{d-2} & \dots & \alpha_{q}^{d-2}
	   \end{pmatrix}
	   \begin{pmatrix}
	    m'_1 \\ \vdots \\ m'_{q}
	   \end{pmatrix}
   = 0.
 \end{equation}

To proceed further the proof of Lemma~\ref{lem3.5}, 
we make use of the following lemma.

\begin{lemma}\label{lem3.6}
 Let $\ell', q, d'_1,\dots,d'_q; \alpha_1, \dots, \alpha_q$ be as in Lemma~{\rm \ref{lem3.5}}.
 Then the linear map
 \begin{equation}\label{eq3.9}
   A_{\ell'-2}^{d-2}(0) \prod_{u=1}^{q}\left( -\alpha_u I_{d-2} +N_{d-2} \right)^{d'_u} : 
   \mathbb{C}^{d-2} \to \mathbb{C}^{\ell'-2}
 \end{equation}
 is surjective. Moreover the basis of its kernel consists of the column vectors of 
 $A_{d-2}^{d'_1}\left(\alpha_1\right)$, $\dots$, $A_{d-2}^{d'_{q}}\left(\alpha_{q}\right)$.
\end{lemma}

\begin{proof}[Proof of Lemma~{\rm \ref{lem3.6}}]
 We first show the surjectivity of the map~(\ref{eq3.9}). 
 Note that $A_{\ell'-2}^{d-2}(0) = \left( I_{\ell' - 2}, O \right)$,
 where $O$ is the zero matrix of size $(\ell'-2, d-\ell')$.
 Hence the map $A_{\ell'-2}^{d-2}(0) : \mathbb{C}^{d-2} \to \mathbb{C}^{\ell'-2}$ is surjective.
 If $\alpha_u \ne 0$,
 then the map $-\alpha_u I_{d-2} +N_{d-2} : \mathbb{C}^{d-2} \to \mathbb{C}^{d-2}$ is bijective.
 If, for instance, $\alpha_1 = 0$, then $\alpha_u \ne 0$ holds for $2 \leq u \leq q$, and
 the map $A_{\ell'-2}^{d-2}(0) (N_{d-2})^{d'_1} = \left( O_1, I_{\ell' - 2}, O_2 \right) :
 \mathbb{C}^{d-2} \to \mathbb{C}^{\ell'-2}$ is surjective, where $O_1$ and $O_2$ are the zero matrices
 of sizes $(\ell' - 2, d'_1)$ and $(\ell' - 2, d-\ell' - d'_1) = (\ell' - 2, \sum_{u=2}^{q}d'_u)$
 respectively.
 Hence in every case, the map~(\ref{eq3.9}) is surjective.

 Since the square matrix
 $\left( A_{d-\ell'}^{d'_1}\left(\alpha_1\right), \dots, A_{d-\ell'}^{d'_{q}}\left(\alpha_{q}\right) \right)$
 is invertible by Proposition~\ref{prop2.4},
 the column vectors of 
 $A_{d-2}^{d'_1}\left(\alpha_1\right), \dots, A_{d-2}^{d'_{q}}\left(\alpha_{q}\right)$
 are linearly independent since $\ell \geq 2$, and span a $\sum_{u=1}^{q}d'_u  = d - \ell'$ dimensional linear subspace
 in $\mathbb{C}^{d-2}$.
 On the other hand, since the map~(\ref{eq3.9}) is surjective, its kernel is a $(d-2) - (\ell' - 2) = d-\ell'$
 dimensional linear subspace in $\mathbb{C}^{d-2}$.
 Hence to complete the proof of Lemma~\ref{lem3.6},
 we only need to check the equality
 \[
    A_{\ell'-2}^{d-2}(0) \left( \prod_{u=1}^{q}\left( -\alpha_u I_{d-2} +N_{d-2} \right)^{d'_u} \right)
    A_{d-2}^{d'_v}\left(\alpha_v\right) = O
 \]
 for every $1 \leq v \leq q$.
 Without loss of generality, it suffices to show
 \begin{equation}\label{eq3.10}
    A_{\ell'-2}^{d-2}(0) \left( \prod_{u=1}^{q}\left( -\alpha_u I_{d-2} +N_{d-2} \right)^{d'_u} \right)
    A_{d-2}^{d'_1}\left(\alpha_1\right) = O.
 \end{equation}

 Since $\left( -\alpha_1 I_{d-2} +N_{d-2} \right)^{d'_1}
 = \sum_{h=0}^{d'_1}\binom{d'_1}{h}(-\alpha_1)^{d'_1-h}(N_{d-2})^h$,
 the $(i,i+h)$-th entry of \\
 $\left( -\alpha_1 I_{d-2} +N_{d-2} \right)^{d'_1}$ is 
 $\binom{d'_1}{h}(-\alpha_1)^{d'_1-h}$ for $0 \leq h \leq \min\{ d'_1, d-2-i \}$,
 and its other entries are zero.
 Hence for $1 \leq i \leq d-2-d'_1$ and $1 \leq j \leq d'_1$, we have
 \begin{align*}
  &\left[ \text{the }
  (i, j)\text{-th entry of }\left( -\alpha_1 I_{d-2} +N_{d-2} \right)^{d'_1} A_{d-2}^{d'_1}\left(\alpha_1\right)
  \right]\\
  =\ 
  &\sum_{h=0}^{d'_1}\left(
	\left[ \text{the } (i, i+h)\text{-th entry of } \left( -\alpha_1 I_{d-2} +N_{d-2} \right)^{d'_1} \right] \right.\\
	&\hspace*{100pt}
	\left.\!\cdot\!
	\left[ \text{the } (i+h, j)\text{-th entry of } A_{d-2}^{d'_1}\left(\alpha_1\right) \right]
	\right)\\
  =\ 
  &\sum_{h=0}^{d'_1} \left\{
	\binom{d'_1}{h}(-\alpha_1)^{d'_1-h} \cdot \binom{i+h-1}{j-1}\alpha_1^{i+h-j}
	\right\}\\
  =\ 
  &(-\alpha_1)^{d'_1 + i - j}\sum_{h=0}^{d'_1}\binom{d'_1}{h}\binom{i+h-1}{j-1} (-1)^{i+h -j}\\
  =\ 
  &(-\alpha_1)^{d'_1 + i - j} \left.\left[ \frac{1}{(j-1)!}\left( \frac{d}{dx} \right)^{j-1}
	\left\{ \sum_{h=0}^{d'_1}\binom{d'_1}{h} x^{i+h-1} \right\} \right]\right|_{x=-1}\\
  =\ 
  &(-\alpha_1)^{d'_1 + i - j} \left.\left[ \frac{1}{(j-1)!}\left( \frac{d}{dx} \right)^{j-1}
	\left\{ x^{i-1}(x+1)^{d'_1} \right\} \right]\right|_{x=-1}\\
  =\ &0
 \end{align*}
 since $j-1 < d'_1$.
 On the other hand, since $\prod_{u=2}^{q}\left( -\alpha_u I_{d-2} +N_{d-2} \right)^{d'_u}$ is a polynomial
 of $N_{d-2}$ with degree $d'_2+\dots+d'_q= d-\ell' - d'_1$, 
 its $(i,j)$-th entry is zero for $j-i > d-\ell' - d'_1$.
 Hence for $1 \leq i \leq \ell' - 2$ and $d-2-d'_1 < j \leq d-2$, we have $j-i > d-\ell' - d'_1$,
 which implies that the $(i,j)$-th entry of
 $A_{\ell'-2}^{d-2}(0) \prod_{u=2}^{q}\left( -\alpha_u I_{d-2} +N_{d-2} \right)^{d'_u}$ is zero.

 To summarize, for every $1 \leq i \leq \ell'-2$ and $1 \leq j \leq d'_1$, we have
 \[
    \left[ \text{the }
  (k, j)\text{-th entry of } \left( -\alpha_1 I_{d-2} +N_{d-2} \right)^{d'_1} A_{d-2}^{d'_1}\left(\alpha_1\right)
  \right] = 0
 \]
 for $1 \leq k \leq d-2-d'_1$, and also have
 \[
    \left[ \text{the }
  (i, k)\text{-th entry of } A_{\ell'-2}^{d-2}(0) \prod_{u=2}^{q}\left( -\alpha_u I_{d-2} +N_{d-2} \right)^{d'_u}
  \right] = 0
 \]
 for $d-2-d'_1 < k \leq d-2$.
 We therefore have the equality~(\ref{eq3.10}), which completes the proof of Lemma~\ref{lem3.6}.
\end{proof}

We return to the proof of Lemma~\ref{lem3.5}.

Suppose first the condition~(1) in Lemma~\ref{lem3.5}.
Then there exist $m_{u,k} \in \mathbb{C}$ for $1 \leq u \leq q$ and $1 \leq k\leq d'_u$
such that the equality~(\ref{eq3.8}) holds.
Multiplying $A_{\ell'-2}^{d-2}(0) \prod_{u=1}^{q}\left( -\alpha_u I_{d-2} +N_{d-2} \right)^{d'_u}$
to both sides of the equality~(\ref{eq3.8}) from the left,
we have the equality~(\ref{eq3.6.2}) by Lemma~\ref{lem3.6}.

Suppose next the condition~(2) in Lemma~\ref{lem3.5}.
Then the vector
\[
	    \left(X_{d-2}\right)^{-1}
	   \begin{pmatrix}
	    \alpha_1&  \dots & \alpha_{q} \\
	    \vdots & \ddots & \vdots\\
	    \alpha_1^{d-2} & \dots & \alpha_{q}^{d-2}
	   \end{pmatrix}
	   \begin{pmatrix}
	    m'_1 \\ \vdots \\ m'_{q}
	   \end{pmatrix}
\]
is contained in the kernel of the linear map 
$A_{\ell'-2}^{d-2}(0) \prod_{u=1}^{q}\left( -\alpha_u I_{d-2} +N_{d-2} \right)^{d'_u}$.
Hence by Lemma~\ref{lem3.6}, 
there exist $m'_{u,k} \in \mathbb{C}$ for $1 \leq u \leq q$ and $1 \leq k\leq d'_u$
such that the equality
\begin{equation}\label{eq3.11}
	    \left(X_{d-2}\right)^{-1}
	   \begin{pmatrix}
	    \alpha_1&  \dots & \alpha_{q} \\
	    \vdots & \ddots & \vdots\\
	    \alpha_1^{d-2} & \dots & \alpha_{q}^{d-2}
	   \end{pmatrix}
	   \begin{pmatrix}
	    m'_1 \\ \vdots \\ m'_{q}
	   \end{pmatrix}
   = \sum_{u=1}^{q} A_{d-2}^{d'_u}\left(\alpha_u\right) 
	\begin{pmatrix} m'_{u,1}\\ \vdots \\ m'_{u,d'_u} \end{pmatrix}
\end{equation}
holds.
Putting $m_{u,k} := -km'_{u,k}$ for $1 \leq u \leq q$ and $1 \leq k\leq d'_u$,
we find that the equality~(\ref{eq3.11}) is the same as the equality~(\ref{eq3.8}),
which is also equivalent to~(\ref{eq3.2.2}).
Hence the implication $(2) \Rightarrow (1)$ is proved,
which completes the proof of Lemma~\ref{lem3.5}.
\end{proof}

\vspace{10pt}

In the rest of the proof of Main Theorem, we continue to consider an analogue 
of each step in Sections~4 and~6 in~\cite{sugi1} and in Section~5 in~\cite{sugi2},
in the case of having multiple fixed points.
However in our case, its proofs are much complicated, compared with the case of 
having no multiple fixed points.

Based on the propositions above, we make the following definition:

\begin{definition}
 We put
 \begin{align*}
  \widetilde{T}(m) &:= \left\{\left. (\zeta_1, \dots, \zeta_{\ell}) \in \mathbb{C}^{\ell} \ \right| \ 
	\text{The equality~(\ref{eq3.6}) holds}  \right\},\\
  \widetilde{S}(m) &:= \left\{\left. (\zeta_1, \dots, \zeta_{\ell}) \in \widetilde{T}(m) \ \right| \ 
	\zeta_1, \dots, \zeta_{\ell} \text{ are mutually distinct}\right\},\\
  \widetilde{B}(m) &:= \widetilde{T}(m) \setminus \widetilde{S}(m),\\
  \mathfrak{S}(m) &:= \left\{ \sigma \in \mathfrak{S}_{\ell} \mid (d_{\sigma(i)}, m_{\sigma(i)}) = (d_i, m_i) \ 
		\text{ for every } 1 \leq i \leq \ell \right\} \qquad \text{and}\\
  \mathfrak{I}'(m) &:= \left\{ \left\{I_1,\ldots,I_q\right\}\ \left| \ 
	 \begin{matrix}
	   q \ge 1,\ \ \emptyset \ne I_u \subseteq \{1,\ldots, \ell \} \textrm{ for every } 1\le u\le q,\\
	   I_1 \amalg \cdots \amalg I_q = \{1,\ldots, \ell \},\\
	   \sum_{i \in I_u} m_i = 0 \textrm{ for every } 1\le u\le q
	 \end{matrix}
	 \right.\right\},
 \end{align*}
 where $\mathfrak{S}_{\ell}$ denotes the $\ell$-th symmetric group,
 and $I_1 \amalg \cdots \amalg I_q$ denotes the disjoint union of $I_1,\dots,I_q$.
 Moreover we put
 \begin{align*}
   \widetilde{E}(\mathbb{I}) &:= \left\{\left. (\zeta_1, \dots, \zeta_{\ell}) \in \mathbb{C}^{\ell} \ \right| \
	i,j \in I \in \mathbb{I} \Rightarrow \zeta_i = \zeta_j \right\} 
	\quad \text{ for } \quad \mathbb{I} \in \mathfrak{I}'(m).
 \end{align*}
\end{definition}

Combining Propositions~\ref{prop3.1},~\ref{prop3.4} and~\ref{prop3.5},
we obviously have the following:

\begin{proposition}\label{prop3.9}
 The equality
 \[
    \widetilde{S}(m) = \left\{ (\zeta_1, \dots, \zeta_{\ell}) \in \mathbb{C}^{\ell} \ \left| \ \begin{matrix}
	\text{There exists a unique non-zero complex number } \rho \text{ such that} \\
	f(z) = z+\rho (z - \zeta_1)^{d_1} \cdots (z - \zeta_{\ell})^{d_{\ell}} \in {\rm Poly}_d(d_1,\dots,d_{\ell}) \text{ holds}\\
	\text{and the equalities } \iota(f,\zeta_i) = m_i \text{ hold for } 1 \leq i \leq \ell.
    \end{matrix}\right.\right\}
 \]
 holds.
\end{proposition}

The group $\mathrm{Aut}(\mathbb{C})$ naturally acts on $\mathbb{C}^{\ell}$ by 
$\gamma \cdot (\zeta_1,\dots,\zeta_{\ell}) := (\gamma(\zeta_1),\dots,\gamma(\zeta_{\ell}))$.
By Proposition~\ref{prop3.9}, we obviously have the following:

\begin{proposition}\label{prop3.10}
 Let $P: \mathrm{Poly}_d \to \mathrm{MP}_d$ be the natural projection.
 Then
 \begin{enumerate}
  \item we can define the surjection
	$\widetilde{\pi}(m) : \widetilde{S}(m) \to \left(\Phi_d(d_1,\dots,d_{\ell})\circ P\right)^{-1}(\overline{m})$ by
	\[
	   (\zeta_1, \dots, \zeta_{\ell}) \mapsto 
	   f(z):= z + \rho(z-\zeta_1)^{d_1}\cdots(z-\zeta_{\ell})^{d_{\ell}},
	\]
	where $\rho$ is defined to be a unique non-zero complex number 
	such that the equalities $\iota(f,\zeta_i) = m_i$ hold for $1 \leq i \leq \ell$.
  \item $\widetilde{S}(m)$ is invariant under the action of $\mathrm{Aut}(\mathbb{C})$ on $\mathbb{C}^{\ell}$.
  \item The actions of $\mathrm{Aut}(\mathbb{C})$ 
	on $\widetilde{S}(m)$ and on $\left(\Phi_d(d_1,\dots,d_{\ell})\circ P\right)^{-1}(\overline{m})$
	commute with the map $\widetilde{\pi}(m)$.
  \item The group $\mathfrak{S}(m)$ acts on $\widetilde{S}(m)$ by the permutation of coordinates.
	Moreover for $\zeta, \zeta' \in \widetilde{S}(m)$, the equality $\widetilde{\pi}(m)(\zeta) = \widetilde{\pi}(m)(\zeta')$ holds if and only if there exists $\sigma \in \mathfrak{S}(m)$ such that $\sigma \cdot \zeta = \zeta'$.
 \end{enumerate}
\end{proposition}

Note that the action of $\mathfrak{S}(m)$ on $\widetilde{S}(m)$ and 
the action of $\mathrm{Aut}(\mathbb{C})$ on $\widetilde{S}(m)$ commute.
Hence the action of $\mathfrak{S}(m)$ on $\widetilde{S}(m)$ naturally induces 
the action of $\mathfrak{S}(m)$ on $\widetilde{S}(m) / \mathrm{Aut}(\mathbb{C})$.
By Proposition~\ref{prop3.10}, we naturally have the following:

\begin{proposition}\label{prop3.11}
 Under the same notation as in Proposition~{\rm \ref{prop3.10}}, 
 the map $\widetilde{\pi}(m)$ induces the surjection
 $\widetilde{\pi}'(m) : \widetilde{S}(m) / \mathrm{Aut}(\mathbb{C}) \to \Phi_d(d_1,\dots,d_{\ell})^{-1}(\overline{m})$,
 which also induces the bijection 
 $\overline{\widetilde{\pi}'(m)} : \left(\widetilde{S}(m) / \mathrm{Aut}(\mathbb{C})\right) / \mathfrak{S}(m)
 \cong \Phi_d(d_1,\dots,d_{\ell})^{-1}(\overline{m})$.
\end{proposition}

On the other hand, we have the following for $\widetilde{B}(m)$:

\begin{proposition}\label{prop3.12}
 We have
 \[
    \widetilde{B}(m) =  \bigcup_{\mathbb{I} \in \mathfrak{I}'(m)} \widetilde{E}(\mathbb{I}).
 \]
\end{proposition}

\begin{proof}
 We prove 
 $\widetilde{B}(m) \supseteq \bigcup_{\mathbb{I} \in \mathfrak{I}'(m)} \widetilde{E}(\mathbb{I})$ first.
 For arbitrary $\mathbb{I} \in \mathfrak{I}'(m)$
 and $\zeta=(\zeta_1,\dots, \zeta_{\ell}) \in \widetilde{E}(\mathbb{I})$,
 we put $\mathbb{I} =: \left\{ I_1,\dots, I_q \right\}$ and
 $\alpha_u := \zeta_i$ for $i\in I_u$ for each $1 \leq u \leq q$.
 Then we have 
 $
  \sum_{i=1}^{\ell} m_i \zeta_i^k
	= \sum_{u=1}^q\sum_{i \in I_u} m_i \zeta_i^k
	= \sum_{u=1}^q\left( \sum_{i \in I_u} m_i \right) \alpha_u^k = 0$ for every $k \in \mathbb{N}$,
 which implies 
 \[
	   \begin{pmatrix}
	    \zeta_1&  \dots & \zeta_{\ell} \\
	    \vdots & \ddots & \vdots\\
	    \zeta_1^{d-2} & \dots & \zeta_{\ell}^{d-2}
	   \end{pmatrix}
	   \begin{pmatrix}
	    m_1 \\ \vdots \\ m_{\ell}
	   \end{pmatrix}
	=
	   \begin{pmatrix}
	    0 \\ \vdots \\ 0
	   \end{pmatrix},
 \]
 and hence $\zeta \in \widetilde{T}(m)$.
 Moreover since $(m_1,\dots,m_{\ell}) \ne (0,\dots,0)$, there exists $I \in \mathbb{I}$ with $\# I \geq 2$,
 which implies $\zeta \notin \widetilde{S}(m)$.
 We therefore have $\zeta \in \widetilde{B}(m)$
 for every $\zeta \in \bigcup_{\mathbb{I} \in \mathfrak{I}'(m)} \widetilde{E}(\mathbb{I})$,
 which assures
 $\widetilde{B}(m) \supseteq \bigcup_{\mathbb{I} \in \mathfrak{I}'(m)} \widetilde{E}(\mathbb{I})$.

We prove 
 $\widetilde{B}(m) \subseteq \bigcup_{\mathbb{I} \in \mathfrak{I}'(m)} \widetilde{E}(\mathbb{I})$
 next.
 For $\zeta=(\zeta_1,\dots, \zeta_{\ell}) \in \widetilde{B}(m)$, we put
 \[
    \mathbb{I}(\zeta) := \left\{ I \ \left| \ \begin{matrix}
	\emptyset \ne I \subseteq \{1, \dots, \ell \},\\
	i,j\in I \Rightarrow \zeta_i=\zeta_j,\\
	i \in I, j \in \{1, \dots, \ell \}\setminus I \Rightarrow \zeta_i \ne \zeta_j
    \end{matrix}\right.\right\},
 \]
 $\mathbb{I}(\zeta) =: \left\{ I_1,\dots, I_q \right\}$, $d'_u = \sum_{i \in I_u}\left( d_i - 1 \right)$ and 
 $\alpha_u := \zeta_i$ for $i\in I_u$ for each $1 \leq u \leq q$.
 Then by definition, we have $I_1 \amalg \dots \amalg I_q = \{1, \dots, \ell \}$, 
 $\sum_{u=1}^q d'_u
 = d - \ell$,
 the mutual distinctness of $\alpha_1, \dots, \alpha_{\ell}$
 and $\sum_{i=1}^{\ell} m_i \zeta_i^k = \sum_{u=1}^q\sum_{i \in I_u} m_i \zeta_i^k
	= \sum_{u=1}^q\left( \sum_{i \in I_u} m_i \right) \alpha_u^k$ for every $k \in \mathbb{N}$.
 Moreover since $\zeta \notin \widetilde{S}(m)$, we have $q < \ell$.
 Hence for $\zeta=(\zeta_1,\dots, \zeta_{\ell}) \in \widetilde{B}(m)$, we have
 \begin{align*}
   0 &= A_{\ell-2}^{d-2}(0) \left(\prod_{i=1}^{\ell}\left( -\zeta_i I_{d-2} +N_{d-2} \right)^{d_i - 1}\right)
	    \left(X_{d-2}\right)^{-1}
	   \begin{pmatrix}
	    \zeta_1&  \dots & \zeta_{\ell} \\
	    \vdots & \ddots & \vdots\\
	    \zeta_1^{d-2} & \dots & \zeta_{\ell}^{d-2}
	   \end{pmatrix}
	   \begin{pmatrix}
	    m_1 \\ \vdots \\ m_{\ell}
	   \end{pmatrix}\\
   &= A_{\ell-2}^{d-2}(0) \left(\prod_{u=1}^{q}\left( -\alpha_u I_{d-2} +N_{d-2} \right)^{d'_u}\right)
	    \left(X_{d-2}\right)^{-1}
	   \begin{pmatrix}
	    \alpha_1&  \dots & \alpha_{q} \\
	    \vdots & \ddots & \vdots\\
	    \alpha_1^{d-2} & \dots & \alpha_{q}^{d-2}
	   \end{pmatrix}
	   \begin{pmatrix}
	    \sum_{i \in I_1}m_i \\ \vdots \\ \sum_{i \in I_q}m_i
	   \end{pmatrix}.
 \end{align*}
 Hence by applying Lemma~\ref{lem3.5} in the case $\ell' = \ell$ and $m'_u = \sum_{i \in I_u}m_i$ for $1 \leq u \leq q$, 
 we have the equality
 \[
    \sum_{u=1}^q A_{d-1}^{d'_u+1}\left( \alpha_u \right) 
	\begin{pmatrix} \sum_{i \in I_u}m_i \\ m_{u,1} \\ \vdots \\ m_{u, d'_u} \end{pmatrix}
    = 0
 \]
 for some $m_{u, k} \in \mathbb{C}$ with $1 \leq u \leq q$ and $1 \leq k \leq d'_u$.
 Since $\sum_{u=1}^q \left( d'_u + 1 \right) = d - \ell + q \leq d-1$,
 the invertibility of
 the square matrix
 $\left( A_{d-\ell+q}^{d'_1+1}\left(\alpha_1\right), \dots, A_{d-\ell+q}^{d'_q+1}\left(\alpha_q\right) \right)$
 implies 
 $^t\!\left( \sum_{i \in I_u}m_i, m_{u,1}, \dots, m_{u, d'_u} \right) = ^t\!(0,0,\dots,0)$
 for every $1 \leq u \leq q$.
 We therefore have $\mathbb{I}(\zeta) \in \mathfrak{I}'(m)$
 and also have $\zeta \in \widetilde{E}\left(\mathbb{I}(\zeta)\right)$,
 which completes the proof of
 $\widetilde{B}(m) \subseteq \bigcup_{\mathbb{I} \in \mathfrak{I}'(m)} \widetilde{E}(\mathbb{I})$.
\end{proof}

By Propositions~\ref{prop3.10}(2) and~\ref{prop3.12}, we have the following:

\begin{proposition}\label{prop3.13}
 The subsets $\widetilde{T}(m), \widetilde{S}(m), \widetilde{B}(m)$ and
 $\widetilde{E}(\mathbb{I})$ for $\mathbb{I} \in \mathfrak{I}'(m)$
 are invariant under the action of $\mathrm{Aut}(\mathbb{C})$ on $\mathbb{C}^{\ell}$.
\end{proposition}
\begin{proof}
 For each $\mathbb{I} \in \mathfrak{I}'(m)$, the subset $\widetilde{E}(\mathbb{I})$ is trivially invariant 
 under the action of $\mathrm{Aut}(\mathbb{C})$ on $\mathbb{C}^{\ell}$,
 which implies the invariance of $\widetilde{B}(m)$ by Proposition~\ref{prop3.12}.
 Moreover $\widetilde{S}(m)$ and $\widetilde{T}(m)$ are also invariant 
 under the action of $\mathrm{Aut}(\mathbb{C})$ on $\mathbb{C}^{\ell}$
 by Proposition~\ref{prop3.10}(2) and $\widetilde{T}(m) = \widetilde{S}(m) \amalg \widetilde{B}(m)$.
\end{proof}


We denote by $\mathbb{P}^{\ell-2}$ the complex projective space of dimension $\ell-2$.
Note that in the case $\ell=2$, the $0$-dimensional projective space $\mathbb{P}^{\ell-2} = \mathbb{P}^{0}$ consists of one point.
We naturally have the isomorphism
\[
    \left( \mathbb{C}^{\ell} / \mathrm{Aut}(\mathbb{C}) \right) \setminus \left\{\overline{0}\right\}
	\cong \mathbb{P}^{\ell - 2}
    \quad \text{by} \quad 
	\overline{(\zeta_1,\dots,\zeta_{\ell-1},\zeta_{\ell})}
	 \mapsto (\zeta_1-\zeta_{\ell}:\dots: \zeta_{\ell-1}-\zeta_{\ell}),
\]
where $\overline{0} := \left\{ (\zeta, \dots, \zeta) \in \mathbb{C}^{\ell} \mid \zeta \in \mathbb{C} \right\}$
denotes the equivalence class of $0 \in \mathbb{C}^{\ell}$ 
under the action of $\mathrm{Aut}(\mathbb{C})$ on $\mathbb{C}^{\ell}$.
Note that we always have $\mathbb{I}_0 := \{\{1, \dots, \ell\}\} \in \mathfrak{I}'(m)$
and $\widetilde{E}(\mathbb{I}_0) = \overline{0}$.

\begin{definition}
 We put
 \begin{align*}
  \begin{pmatrix} \psi_1(\zeta) \\ \vdots \\ \psi_{\ell-2}(\zeta) \end{pmatrix}
  &:= A_{\ell-2}^{d-2}(0) \left( N_{d-2} \right)^{d_{\ell} - 1}
		\left(\prod_{i=1}^{\ell-1}\left( -\zeta_i I_{d-2} +N_{d-2} \right)^{d_i - 1}\right)\\
	&\hspace{30pt}
	    \cdot \left(X_{d-2}\right)^{-1}
	   \begin{pmatrix}
	    \zeta_1&  \dots & \zeta_{\ell-1} \\
	    \vdots & \ddots & \vdots\\
	    \zeta_1^{d-2} & \dots & \zeta_{\ell-1}^{d-2}
	   \end{pmatrix}
	   \begin{pmatrix}
	    m_1 \\ \vdots \\ m_{\ell-1}
	   \end{pmatrix}
	\quad \text{ for } \ \zeta = (\zeta_1, \dots, \zeta_{\ell-1}) \in \mathbb{C}^{\ell - 1}, \\
  T(m) &:= \left\{\left. (\zeta_1: \dots: \zeta_{\ell-1}) \in \mathbb{P}^{\ell-2} \ \right| \ 
	\psi_1(\zeta) = \dots = \psi_{\ell - 2}(\zeta) = 0
	\right\},\\
  S(m) &:= \left\{\left. (\zeta_1: \dots: \zeta_{\ell-1}) \in T(m) \ \right| \ 
	\zeta_1, \dots, \zeta_{\ell-1} \text{ and } 0 \text{ are mutually distinct}\right\},\\
  B(m) &:= T(m) \setminus S(m) \qquad \text{and}\\
  \mathfrak{I}(m) &:= \left\{ \left\{I_1,\ldots,I_q\right\}\ \left| \ 
	 \begin{matrix}
	   q \ge 2,\ \ \emptyset \ne I_u \subset \{1,\ldots, \ell \} \textrm{ for every } 1\le u\le q,\\
	   I_1 \amalg \cdots \amalg I_q = \{1,\ldots, \ell \},\\
	   \sum_{i \in I_u} m_i = 0 \textrm{ for every } 1\le u\le q
	 \end{matrix}
	 \right.\right\}.
 \end{align*}
 Moreover we put
 \begin{align*}
   E(\mathbb{I}) &:= \left\{\left. (\zeta_1: \dots: \zeta_{\ell-1}) \in \mathbb{P}^{\ell-2} \ \right| \
	i,j \in I \in \mathbb{I} \Rightarrow \zeta_i = \zeta_j, \text{ where } \zeta_{\ell}=0 \right\} 
	\text{ for } \mathbb{I} \in \mathfrak{I}(m).
 \end{align*}
\end{definition}

Note that $^t\!\left( \psi_1(\zeta), \dots, \psi_{\ell-2}(\zeta) \right)$ is obtained
from the left-hand side of the equality~(\ref{eq3.6}) by substituting $\zeta_{\ell} = 0$.
Also 
note that $\mathfrak{I}(m)$ is obtained from $\mathfrak{I}'(m)$
by excluding exactly one element $\mathbb{I}_0 = \{\{1, \dots, \ell\}\} \in \mathfrak{I}'(m)$.
Moreover for every $\mathbb{I} \in \mathfrak{I}(m)$, we have $\# \mathbb{I} \geq 2$,
which implies $E(\mathbb{I}) \ne \emptyset$.
Hence
under the isomorphism
$\left( \mathbb{C}^{\ell} / \mathrm{Aut}(\mathbb{C}) \right) \setminus \left\{\overline{0}\right\}
	\cong \mathbb{P}^{\ell - 2}$, we clearly have
\begin{align*}
 &\left( \widetilde{T}(m) / \mathrm{Aut}(\mathbb{C}) \right) \setminus \left\{\overline{0}\right\}
	\cong T(m), \quad
 \widetilde{S}(m) / \mathrm{Aut}(\mathbb{C}) \cong S(m), \quad
 \left( \widetilde{B}(m) / \mathrm{Aut}(\mathbb{C}) \right) \setminus \left\{\overline{0}\right\}
	\cong B(m),\\
  &\left( \widetilde{E}(\mathbb{I})/\mathrm{Aut}(\mathbb{C}) \right) \setminus \left\{\overline{0}\right\}
	\cong E(\mathbb{I}) \ \text{ for } \ \mathbb{I} \in \mathfrak{I}(m)
 \quad \text{ and } \quad
 B(m) =  \bigcup_{\mathbb{I} \in \mathfrak{I}(m)} E(\mathbb{I})
\end{align*}
by Propositions~\ref{prop3.12} and~\ref{prop3.13}.

\begin{proposition}\label{prop3.15}
 $\psi_k(\zeta)$ is a homogeneous polynomial of $\zeta_1,\dots, \zeta_{\ell-1}$
 with degree $d-\ell + k$
  for $1 \leq k \leq \ell-2$.
\end{proposition}

\begin{proof}
Direct calculation easily verifies the proposition.
\end{proof}

By Proposition~\ref{prop3.11} and the isomorphism
$\widetilde{S}(m) / \mathrm{Aut}(\mathbb{C}) \cong S(m)$,
we have the following:


\begin{proposition}\label{prop3.16}
 Under the isomorphism $\widetilde{S}(m) / \mathrm{Aut}(\mathbb{C}) \cong S(m)$,
 the surjection
 $\widetilde{\pi}'(m) : \widetilde{S}(m) / \mathrm{Aut}(\mathbb{C}) \to \Phi_d(d_1,\dots,d_{\ell})^{-1}(\overline{m})$
 induces the surjection
 \[
   \pi(m) : S(m) \to \Phi_d(d_1,\dots,d_{\ell})^{-1}(\overline{m}).
 \]
 Moreover the action of $\mathfrak{S}(m)$ on $\widetilde{S}(m) / \mathrm{Aut}(\mathbb{C})$
 induces the action of $\mathfrak{S}(m)$ on $S(m)$, and
 the bijection
 $\overline{\widetilde{\pi}'(m)} : \left(\widetilde{S}(m) / \mathrm{Aut}(\mathbb{C})\right) / \mathfrak{S}(m)
 \cong \Phi_d(d_1,\dots,d_{\ell})^{-1}(\overline{m})$
 induces the bijection
 \[
   \overline{\pi(m)} : S(m)/\mathfrak{S}(m) \cong \Phi_d(d_1,\dots,d_{\ell})^{-1}(\overline{m}).
 \]
\end{proposition}

\begin{definition}\label{def3.18}
 We put
 \begin{align*}
  \Sigma(m) &:= \left\{ (\zeta_1: \dots: \zeta_{\ell}) \in \mathbb{P}^{\ell-1} \ \left| \ 
	 (\zeta_1, \dots, \zeta_{\ell}) \in \widetilde{S}(m) \quad \text{and} \quad 
	 d_1\zeta_1 + \dots + d_{\ell}\zeta_{\ell}=0
	\right.\right\} \quad \text{and}\\
  \widetilde{\Sigma}(m) &:= \left\{ (\zeta_1, \dots, \zeta_{\ell}) \in \widetilde{S}(m) \ \left| \ 
	\widetilde{\pi}(m)(\zeta_1, \dots, \zeta_{\ell}) \in {\rm MC}_d(d_1,\dots,d_\ell)
	\right.\right\},
 \end{align*}
 where $\widetilde{\pi}(m)$ is the map defined in Proposition~\ref{prop3.10}(1).
\end{definition}

Summing up the propositions above, we naturally have the following:

\begin{proposition}\label{prop3.19} \ 
 \begin{enumerate}
  \item The natural map $\Sigma(m) \to S(m)$ 
	defined by $(\zeta_1: \dots: \zeta_{\ell}) \mapsto (\zeta_1 - \zeta_{\ell}: \dots: \zeta_{\ell-1} - \zeta_{\ell})$
	is well-defined and bijective.
  \item The group $\mathfrak{S}(m)$ acts on $\Sigma(m)$ by the permutation of coordinates.
	Hence we have the bijection
	$\Sigma(m)/\mathfrak{S}(m) \cong S(m)/\mathfrak{S}(m) \cong \Phi_d(d_1,\dots,d_{\ell})^{-1}(\overline{m})$.
  \item The map $\widetilde{\Sigma}(m) \to \widehat{\Phi}_d(d_1,\dots,d_\ell)^{-1}(\overline{m})$
	defined by $(\zeta_1, \dots, \zeta_{\ell}) \mapsto f(z) = z + (z - \zeta_1)^{d_1} \cdots (z - \zeta_{\ell})^{d_{\ell}}$
	is well-defined and surjective.
  \item The group $\mathfrak{S}(m)$ acts on $\widetilde{\Sigma}(m)$ freely by the permutation of coordinates.
	Moreover we have the bijection 
	$\widetilde{\Sigma}(m)/\mathfrak{S}(m) \cong \widehat{\Phi}_d(d_1,\dots,d_\ell)^{-1}(\overline{m})$.
	Hence we have \\
	$\#\widehat{\Phi}_d(d_1,\dots,d_\ell)^{-1}(\overline{m}) = \#\widetilde{\Sigma}(m) / \#\mathfrak{S}(m)$.
  \item The natural projection $\widetilde{\Sigma}(m) \to \Sigma(m)$ defined by 
	$(\zeta_1, \dots, \zeta_{\ell}) \mapsto (\zeta_1: \dots: \zeta_{\ell})$ is well-defined and surjective.
	Moreover it is a $(d-1)$-to-one map.
	Hence we have $\# \widetilde{\Sigma}(m) = (d-1) \!\cdot\! \#\Sigma(m)$.
 \end{enumerate}
\end{proposition}
\begin{proof}
 For every $(\zeta_1, \dots, \zeta_{\ell}) \in \widetilde{S}(m)$, there exists a unique non-zero complex number $\rho$
 such that $f(z) = z+\rho (z - \zeta_1)^{d_1} \cdots (z - \zeta_{\ell})^{d_{\ell}} \in {\rm Poly}_d(d_1,\dots,d_{\ell})$ holds
	and that the equalities $\iota(f,\zeta_i) = m_i$ hold for $1 \leq i \leq \ell$, by Proposition~\ref{prop3.9}.
	Putting $\zeta'_i = \zeta_i - b$ for $b := (d_1\zeta_1+\dots+d_{\ell}\zeta_{\ell})/d$,
	we still have $(\zeta'_1,\dots,\zeta'_{\ell}) \in \widetilde{S}(m)$ and also have $(\zeta'_1:\dots:\zeta'_{\ell}) \in \Sigma(m)$.
	Here, note that $z+\rho (z - \zeta'_1)^{d_1} \cdots (z - \zeta'_{\ell})^{d_{\ell}}$ and 
	$z+\rho a^{-(d-1)} (z - a\zeta'_1)^{d_1} \cdots (z - a\zeta'_{\ell})^{d_{\ell}}$
	are affinely conjugate for $a \in \mathbb{C} \setminus \{0\}$.
	It follows that $(a\zeta'_1,\dots, a\zeta'_{\ell})$ belongs to $\widetilde{\Sigma}(m)$ if and only if $\rho a^{-(d-1)}=1$.
	Since $\rho \ne 0$, we always have $\#\left\{ a \in \mathbb{C} \mid \rho a^{-(d-1)}=1 \right\} = d-1$, 
	which implies the assertion~(5).
	The rests are almost obvious by the propositions already obtained.
\end{proof}

\begin{proposition}\label{prop3.17}
 In the case $\ell \geq 3$,
 we have
 \[
    \left. \det \begin{pmatrix}
	\frac{\partial \psi_1}{\partial \zeta_1} & \dots & \frac{\partial \psi_1}{\partial \zeta_{\ell-2}} \\
	\vdots 							& \ddots & \vdots \\
	\frac{\partial \psi_{\ell-2}}{\partial \zeta_1} & \dots & \frac{\partial \psi_{\ell-2}}{\partial \zeta_{\ell-2}}
   \end{pmatrix} \right|_{\zeta_{\ell - 1} = 1}
   \ne 0
 \]
 for every $(\zeta_1: \dots: \zeta_{\ell-1}) \in S(m)$.
\end{proposition}

\begin{proof}
 For $1 \leq i \leq \ell - 2$, we have
 \begin{align*}
  \begin{pmatrix} \frac{\partial\psi_1}{\partial \zeta_i} \\ \vdots \\ \frac{\partial\psi_{\ell-2}}{\partial \zeta_i}
  \end{pmatrix}
   = \ &\frac{\partial}{\partial \zeta_i}\left\{
	A_{\ell-2}^{d-2}(0) \left( N_{d-2} \right)^{d_{\ell} - 1}
		\left(\prod_{j=1}^{\ell-1}\left( -\zeta_j I_{d-2} +N_{d-2} \right)^{d_j - 1}\right)
		\begin{matrix} \ \\ \ \\ \ \\ \ \end{matrix}\right. \\ 
	&\hspace*{170pt} \left.  \begin{matrix} \ \\ \ \\ \ \\ \ \end{matrix}
	    \cdot \left(X_{d-2}\right)^{-1}
	   \begin{pmatrix}
	    \zeta_1&  \dots & \zeta_{\ell-1} \\
	    \vdots & \ddots & \vdots\\
	    \zeta_1^{d-2} & \dots & \zeta_{\ell-1}^{d-2}
	   \end{pmatrix}
	   \begin{pmatrix}
	    m_1 \\ \vdots \\ m_{\ell-1}
	   \end{pmatrix}
	\right\}\\
  = \ &A_{\ell-2}^{d-2}(0) \left( N_{d-2} \right)^{d_{\ell} - 1}
		\left(\prod_{j=1}^{\ell-1}\left( -\zeta_j I_{d-2} +N_{d-2} \right)^{d_j - 1}\right) \left(X_{d-2}\right)^{-1}
	\begin{pmatrix} 1 \\ 2\zeta_i \\ \vdots \\ (d-2)\zeta_i^{d-3} \end{pmatrix} \left( m_i \right)\\
    &+ A_{\ell-2}^{d-2}(0) \left( N_{d-2} \right)^{d_{\ell} - 1}
		\left(\prod_{ 1 \leq j \leq \ell-1,\ j \ne i }
			\left( -\zeta_j I_{d-2} +N_{d-2} \right)^{d_j - 1}\right)\\
	& \hspace*{60pt}\cdot (d_i - 1) \left( -\zeta_i I_{d-2} +N_{d-2} \right)^{d_i - 2}
			\left\{  \frac{\partial}{\partial \zeta_i} \left( -\zeta_i I_{d-2} +N_{d-2} \right)\right\} \\
	& \hspace*{185pt} \cdot \left(X_{d-2}\right)^{-1}
	   \begin{pmatrix}
	    \zeta_1&  \dots & \zeta_{\ell-1} \\
	    \vdots & \ddots & \vdots\\
	    \zeta_1^{d-2} & \dots & \zeta_{\ell-1}^{d-2}
	   \end{pmatrix}
	   \begin{pmatrix}
	    m_1 \\ \vdots \\ m_{\ell-1}
	   \end{pmatrix}\\
  = \ &m_i A_{\ell-2}^{d-2}(0) \left( N_{d-2} \right)^{d_{\ell} - 1}
		\left(\prod_{j=1}^{\ell-1}\left( -\zeta_j I_{d-2} +N_{d-2} \right)^{d_j - 1}\right)
	\begin{pmatrix} 1 \\ \zeta_i \\ \vdots \\ \zeta_i^{d-3} \end{pmatrix}\\
    &- (d_i - 1)  A_{\ell-2}^{d-2}(0) \left( N_{d-2} \right)^{d_{\ell} - 1}
		\left(\prod_{ 1 \leq j \leq \ell-1,\ j \ne i }
			\left( -\zeta_j I_{d-2} +N_{d-2} \right)^{d_j - 1}\right)\\
	&\hspace*{77pt} \cdot \left( -\zeta_i I_{d-2} +N_{d-2} \right)^{d_i - 2}
		\left(X_{d-2}\right)^{-1}
	   \begin{pmatrix}
	    \zeta_1&  \dots & \zeta_{\ell-1} \\
	    \vdots & \ddots & \vdots\\
	    \zeta_1^{d-2} & \dots & \zeta_{\ell-1}^{d-2}
	   \end{pmatrix}
	   \begin{pmatrix}
	    m_1 \\ \vdots \\ m_{\ell-1}
	   \end{pmatrix}.
 \end{align*}
 Moreover since 
 \[
    \left( -\zeta_j I_{d-2} +N_{d-2} \right) \cdot \!\, ^t( 1, \zeta_i, \dots, \zeta_i^{k},*,\dots,* )
    =  \!\, ^t(\zeta_i-\zeta_j, (\zeta_i-\zeta_j)\zeta_i,\dots, (\zeta_i-\zeta_j)\zeta_i^{k-1},*,*\dots,*)
 \]
 for $1 \leq k \leq d-3$, we have
 \[
   A_{\ell-2}^{d-2}(0) \left( N_{d-2} \right)^{d_{\ell} - 1}
		\left(\prod_{j=1}^{\ell-1}\left( -\zeta_j I_{d-2} +N_{d-2} \right)^{d_j - 1}\right)
	\begin{pmatrix} 1 \\ \zeta_i \\ \vdots \\ \zeta_i^{d-3} \end{pmatrix}
  = \zeta_i^{d_\ell - 1}\prod_{j=1}^{\ell - 1}(\zeta_i- \zeta_j)^{d_j- 1}
	\begin{pmatrix} 1 \\ \zeta_i \\ \vdots \\ \zeta_i^{\ell-3} \end{pmatrix}.
 \]
 Hence we have
 \begin{lemma}\label{lem3.18}
  If $d_i = 1$, then 
 \[
   \begin{pmatrix} \frac{\partial\psi_1}{\partial \zeta_i} \\ \vdots \\ \frac{\partial\psi_{\ell-2}}{\partial \zeta_i}
   \end{pmatrix}
   = m_i \zeta_i^{d_\ell - 1} \left( \prod_{ 1 \leq j \leq \ell-1,\ j \ne i }(\zeta_i- \zeta_j)^{d_j- 1} \right)
	\begin{pmatrix} 1 \\ \zeta_i \\ \vdots \\ \zeta_i^{\ell-3} \end{pmatrix}.
 \]
 Moreover for $(\zeta_1: \dots: \zeta_{\ell-1}) \in S(m)$,
 we have $m_i \zeta_i^{d_\ell - 1}\prod_{ 1 \leq j \leq \ell-1,\ j \ne i }(\zeta_i- \zeta_j)^{d_j- 1} \ne 0$
 \end{lemma}
\begin{proof} 
 By the assertion~(2) in Proposition~\ref{prop3.2}, 
 the conditions $d_i=1$ and $(\zeta_1: \dots: \zeta_{\ell-1}) \in S(m)$ imply $m_i \ne 0$.
 The rest of the assertions are obvious by the argument above.
\end{proof}
 We consider the case $d_i \geq 2$ next.
 If $d_i \geq 2$, then we have $\prod_{j=1}^{\ell - 1}(\zeta_i- \zeta_j)^{d_j- 1} = 0$, which implies
 \begin{multline*}
   \begin{pmatrix} \frac{\partial\psi_1}{\partial \zeta_i} \\ \vdots \\ \frac{\partial\psi_{\ell-2}}{\partial \zeta_i}
   \end{pmatrix}
   = - (d_i - 1)  A_{\ell-2}^{d-2}(0) \left( N_{d-2} \right)^{d_{\ell} - 1}
		\left(\prod_{ 1 \leq j \leq \ell-1,\ j \ne i }
			\left( -\zeta_j I_{d-2} +N_{d-2} \right)^{d_j - 1}\right)\\
	 \cdot \left( -\zeta_i I_{d-2} +N_{d-2} \right)^{d_i - 2}
		\left(X_{d-2}\right)^{-1}
	   \begin{pmatrix}
	    \zeta_1&  \dots & \zeta_{\ell-1} \\
	    \vdots & \ddots & \vdots\\
	    \zeta_1^{d-2} & \dots & \zeta_{\ell-1}^{d-2}
	   \end{pmatrix}
	   \begin{pmatrix}
	    m_1 \\ \vdots \\ m_{\ell-1}
	   \end{pmatrix}.
 \end{multline*}
 We put
 \begin{multline*}
  \begin{pmatrix}
    \xi_1 \\ \vdots \\ \xi_{d-2}
  \end{pmatrix}
  :=    \left( N_{d-2} \right)^{d_{\ell} - 1}
		\left(\prod_{ 1 \leq j \leq \ell-1,\ j \ne i }
			\left( -\zeta_j I_{d-2} +N_{d-2} \right)^{d_j - 1}\right)\\
	 \cdot \left( -\zeta_i I_{d-2} +N_{d-2} \right)^{d_i - 2}
		\left(X_{d-2}\right)^{-1}
	   \begin{pmatrix}
	    \zeta_1&  \dots & \zeta_{\ell-1} \\
	    \vdots & \ddots & \vdots\\
	    \zeta_1^{d-2} & \dots & \zeta_{\ell-1}^{d-2}
	   \end{pmatrix}
	   \begin{pmatrix}
	    m_1 \\ \vdots \\ m_{\ell-1}
	   \end{pmatrix}.
 \end{multline*}
 Then for $(\zeta_1: \dots: \zeta_{\ell-1}) \in S(m)$,
 we have $^t(\psi_1(\zeta), \dots, \psi_{\ell-2}(\zeta)) =\!\, ^t(0,\dots,0)$, which is equivalent to
 \[
    A_{\ell-2}^{d-2}(0) \left( -\zeta_i I_{d-2} +N_{d-2} \right) \cdot\!\,
    ^t(\xi_1,\dots, \xi_{d-2}) = 0.
 \]
 Hence we have $\xi_k =\xi_1 \zeta_i^{k-1}$ for $1 \leq k \leq \ell - 1$, and therefore we have
 \begin{equation}\label{eq3.12}
    \begin{pmatrix} \frac{\partial\psi_1}{\partial \zeta_i} \\ \vdots \\ \frac{\partial\psi_{\ell-2}}{\partial \zeta_i}
   \end{pmatrix}
   =  - (d_i - 1) A_{\ell-2}^{d-2}(0) 
     \begin{pmatrix}
      \xi_1 \\ \vdots \\ \xi_{d-2}
     \end{pmatrix}
  = - (d_i - 1) \xi_1  \begin{pmatrix} 1 \\ \zeta_i \\ \vdots \\ \zeta_i^{\ell - 3} \end{pmatrix}.
 \end{equation}
 \begin{lemma}\label{lem3.19}
  In the equality~{\rm (\ref{eq3.12})}, we have $\xi_1 \ne 0$.
 \end{lemma}
 \begin{proof}
  Suppose that $\xi_1 = 0$ holds.
  Then we have 
  $\xi_k=0$ for $1 \leq k \leq \ell - 1$,
  which implies that the vector
  \[
     \left(X_{d-2}\right)^{-1}
	   \begin{pmatrix}
	    \zeta_1&  \dots & \zeta_{\ell-1} \\
	    \vdots & \ddots & \vdots\\
	    \zeta_1^{d-2} & \dots & \zeta_{\ell-1}^{d-2}
	   \end{pmatrix}
	   \begin{pmatrix}
	    m_1 \\ \vdots \\ m_{\ell-1}
	   \end{pmatrix}
  \]
  is contained in the kernel of the linear map 
  \[
     A_{\ell-1}^{d-2}(0) \left( N_{d-2} \right)^{d_{\ell} - 1}
		\left(\prod_{ 1 \leq j \leq \ell-1,\ j \ne i }
			\left( -\zeta_j I_{d-2} +N_{d-2} \right)^{d_j - 1}\right)
	  	\left( -\zeta_i I_{d-2} +N_{d-2} \right)^{d_i - 2}.
  \]
  Hence by applying Lemma~\ref{lem3.5} in the case 
  $\ell' = \ell + 1$, $q=\ell$, 
  $d'_j = d_j - 1$ for $j \in \{j \in \mathbb{Z} \mid 1 \leq j \leq \ell,\ j \ne i \}$,
  $d'_i = d_i - 2$,
  $\alpha_i = \zeta_i$ for $1 \leq i \leq \ell - 1$,
  $\alpha_\ell = 0$
  and $m'_i = m_i$ for $1 \leq i \leq \ell$,
  we have the equality
  \[
     \sum_{1 \leq j \leq \ell-1,\ j \ne i} A_{d-1}^{d_j}\left(\zeta_j\right) 
	\begin{pmatrix} m_j \\ m_{j,1} \\ \vdots \\ m_{j,d_j-1} \end{pmatrix}
	+ A_{d-1}^{d_\ell}\left(0\right) \begin{pmatrix} m_{\ell} \\ m_{\ell,1}\\ \vdots \\ m_{\ell,d_\ell-1} \end{pmatrix}
	+ A_{d-1}^{d_i-1}\left(\zeta_i\right) \begin{pmatrix} m_i \\ m_{i,1}\\ \vdots \\ m_{i,d_i-2} \end{pmatrix}
     = 0
  \]
  for some $m_{j,k} \in \mathbb{C}$ with
  \[
    (j,k) \in \left\{ (j,k) \in \mathbb{Z}^2 \ \left| \begin{matrix} 1 \leq j \leq \ell, \\
		j \ne i \Rightarrow 1\leq k \leq d_j - 1, \\
		j = i \Rightarrow 1 \leq k \leq d_i - 2
	\end{matrix} \right. \right\}.
  \]
  Since $\left( \sum_{1 \leq j \leq \ell, \ j \ne i} d_j \right) + (d_i - 1) = d - 1$,
  the square matrix 
  \[
    \left( A_{d-1}^{d_1}\left(\zeta_1\right), \dots, A_{d-1}^{d_i - 1}\left(\zeta_i\right) , \dots, 
    A_{d-1}^{d_{\ell- 1}}\left(\zeta_{\ell - 1}\right) , A_{d-1}^{d_\ell}\left(0\right)  \right)
    \]
  is invertible by Proposition~\ref{prop2.4}.
  We therefore have $(m_1, \dots, m_{\ell}) = (0,\dots, 0)$,
  which contradicts the assumption $(m_1, \dots, m_{\ell}) \in\mathbb{C}^\ell \setminus \{0\}$.
  Hence the contradiction assures $\xi_1 \ne 0$.
 \end{proof}

 By Lemmas~\ref{lem3.18} and~\ref{lem3.19}, we have
 \begin{lemma}\label{lem3.20}
  For every $(\zeta_1: \dots: \zeta_{\ell-1}) \in S(m)$ and for every $1 \leq i \leq \ell - 2$,
  there exists a non-zero complex number $c_i$ such that the equality
  \[
     \begin{pmatrix} \frac{\partial\psi_1}{\partial \zeta_i} \\ \vdots \\ \frac{\partial\psi_{\ell-2}}{\partial \zeta_i}
   \end{pmatrix}
   = c_i \begin{pmatrix} 1 \\ \zeta_i \\ \vdots \\ \zeta_i^{\ell - 3} \end{pmatrix}
  \]
  holds.
 \end{lemma}
 
 Hence by Lemma~\ref{lem3.20}, we have
 \[
	\det \begin{pmatrix}
	\frac{\partial \psi_1}{\partial \zeta_1} & \dots & \frac{\partial \psi_1}{\partial \zeta_{\ell-2}} \\
	\vdots 							& \ddots & \vdots \\
	\frac{\partial \psi_{\ell-2}}{\partial \zeta_1} & \dots & \frac{\partial \psi_{\ell-2}}{\partial \zeta_{\ell-2}}
   \end{pmatrix}
   = \prod_{i=1}^{\ell - 2} c_i \cdot \prod_{1 \leq i < j \leq \ell - 2}(\zeta_j - \zeta_i) \ne 0,
 \]
 which completes the proof of Proposition~\ref{prop3.17}.
\end{proof}

By Propositions~\ref{prop3.15} and~\ref{prop3.17},
we have the following:

\begin{proposition}\label{prop3.21}
 We always have $\# S(m) \leq \frac{(d-2)!}{(d-\ell)!}$.
 Moreover the equality $\# S(m) = \frac{(d-2)!}{(d-\ell)!}$ holds if and only if $\mathfrak{I}(m) = \emptyset$.
\end{proposition}

\begin{proof}
 First we consider the case $\ell=2$. In this case, the equality~(\ref{eq3.6}) always holds, 
 which implies $\widetilde{T}(m) = \mathbb{C}^2$ and 
 $\widetilde{S}(m) = \left\{\left. (\zeta_1,\zeta_2) \in \mathbb{C}^2 \ \right| \ \zeta_1 \ne \zeta_2 \right\}$.
 Hence we always have $T(m) = S(m) = \mathbb{P}^0$, and therefore have $\# S(m) = 1 = \frac{(d-2)!}{(d-\ell)!}$.
 On the other hand, since $m_1 + m_2 = 0$ and $(m_1,m_2) \ne (0,0)$, 
 we have $m = (m_1,-m_1)$ with $m_1 \in \mathbb{C} \setminus \{0\}$,
 which implies that $\mathfrak{I}(m) = \emptyset$ always holds.

 We consider the case $\ell \geq 3$ next.
 In this case, $S(m)$ is a discrete set
 by Proposition~\ref{prop3.17}.
 Moreover for every $\zeta \in S(m)$,
 the intersection multiplicity of $\psi_1(\zeta), \dots, \psi_{\ell - 2}(\zeta)$ at $\zeta$ is $1$.
 Hence by a similar argument to the proof of Proposition~6.2 in~\cite{sugi1},
 we have $\# S(m) \leq \prod_{k=1}^{\ell - 2}\deg \psi_k = \frac{(d-2)!}{(d-\ell)!}$.
 Moreover since $B(m) = T(m) \setminus S(m)$ and
 $B(m) = \bigcup_{\mathbb{I} \in \mathfrak{I}(m)} E(\mathbb{I})$,
 the equality $\# S(m) = \frac{(d-2)!}{(d-\ell)!}$ holds if and only if 
 $\mathfrak{I}(m) = \emptyset$ holds,
 which can also be obtained by a similar argument to the proof of Proposition~6.6 in~\cite{sugi1}.
\end{proof}

Based on the propositions above, we complete the proof of Main Theorem.

\begin{proof}[Proof of Main Theorem]\ 

 First we consider the case $(m_1,\dots,m_\ell) = 0$.
 In this case, we have 
 $\Phi_d(d_1,\dots,d_\ell)^{-1}(\overline{m}) = \widehat{\Phi}_d(d_1,\dots,d_\ell)^{-1}(\overline{m}) = \emptyset$
 by Proposition~\ref{prop3.3}.
 Moreover the condition~(2c) in Main Theorem is not satisfied.
 Hence in this case, Main Theorem holds.
 In the rest of the proof we assume $(m_1,\dots,m_\ell) \ne 0$.
 
 Note first that
 \begin{itemize}
  \item the condition $\sum_{i \in I}m_i \ne 0$ for every $\emptyset \ne I \subsetneq \{1,\dots,\ell\}$
	is equivalent to $\mathfrak{I}(m) = \emptyset$.
  \item $\left(d_1,m_1\right), \dots, \left(d_\ell,m_\ell\right)$ are mutually distinct if and only if $\#\mathfrak{S}(m) = 1$.
 \end{itemize}
 Hence the condition~(2c) in Main Theorem is equivalent 
 to ``$\mathfrak{I}(m) = \emptyset$ and $\#\mathfrak{S}(m) = 1$''.
 
 We consider $\# \widehat{\Phi}_d(d_1,\dots,d_\ell)^{-1}(\overline{m})$ first.
 By the assertions~(1),~(4) and~(5) in Proposition~\ref{prop3.19}, we always have 
 \[
   \#\widehat{\Phi}_d(d_1,\dots,d_\ell)^{-1}(\overline{m}) = \frac{(d-1) \!\cdot\! \#S(m)}{\#\mathfrak{S}(m)}.
 \]
 Hence by Proposition~\ref{prop3.21}, we have 
 \[
   \#\widehat{\Phi}_d(d_1,\dots,d_\ell)^{-1}(\overline{m}) \leq \frac{(d-1)!}{(d-\ell)!}.
 \]
 Moreover 
 the equality $\#\widehat{\Phi}_d(d_1,\dots,d_\ell)^{-1}(\overline{m}) = \frac{(d-1)!}{(d-\ell)!}$ holds
 if and only if $\# S(m) = \frac{(d-2)!}{(d-\ell)!}$ and $\#\mathfrak{S}(m) = 1$, 
 which is equivalent to the condition ``$\mathfrak{I}(m) = \emptyset$ and $\#\mathfrak{S}(m) = 1$''
 by Proposition~\ref{prop3.21}.
 Hence we have the implication (b) $\Leftrightarrow$ (c) in Main Theorem~(2).

 We consider $\# \Phi_d(d_1,\dots,d_\ell)^{-1}(\overline{m})$ next.
 By Propositions~\ref{prop3.16} and~\ref{prop3.21}, we have
 \[
     \Phi_d(d_1,\dots,d_{\ell})^{-1}(\overline{m}) \cong S(m)/\mathfrak{S}(m) \qquad \text{and} \qquad
     \# S(m) \leq \frac{(d-2)!}{(d-\ell)!},
 \]
 which implies the inequality
 \[
   \#\Phi_d(d_1,\dots,d_{\ell})^{-1}(\overline{m}) \leq \frac{(d-2)!}{(d-\ell)!}.
 \]
 On the other hand, 
 the isomorphism $\overline{p} : \mathrm{MC}_d / \left( \mathbb{Z} /(d-1)\mathbb{Z} \right) \cong \mathrm{MP}_d$
 implies the inequality 
 $\# \widehat{\Phi}_d(d_1,\dots,d_\ell)^{-1}(\overline{m}) /(d-1) \leq
 \# \Phi_d(d_1,\dots,d_\ell)^{-1}(\overline{m})$.
 Hence we have 
 \[
   \frac{\# \widehat{\Phi}_d(d_1,\dots,d_\ell)^{-1}(\overline{m})}{d-1} \leq 
   \#\Phi_d(d_1,\dots,d_{\ell})^{-1}(\overline{m}) \leq \frac{(d-2)!}{(d-\ell)!},
 \]
 which assures the implication (b) $\Rightarrow$ (a) in Main Theorem~(2).

 Last of all, we show the implication (a) $\Rightarrow$ (c) in Main Theorem~(2), except in the case $d = \ell = 3$.
 Since $\Phi_d(d_1,\dots,d_{\ell})^{-1}(\overline{m}) \cong S(m)/\mathfrak{S}(m)$ and $\# S(m) \leq \frac{(d-2)!}{(d-\ell)!}$,
 the condition $\#\Phi_d(d_1,\dots,d_{\ell})^{-1}(\overline{m}) = \frac{(d-2)!}{(d-\ell)!}$ is satisfied if and only if
 $\# S(m) = \frac{(d-2)!}{(d-\ell)!}$ holds, and the action of $\mathfrak{S}(m)$ on $S(m)$ is trivial.
 Moreover by Proposition~\ref{prop3.21}, $\# S(m) = \frac{(d-2)!}{(d-\ell)!}$ holds if and only if $\mathfrak{I}(m) = \emptyset$ holds.
 Hence the implication (a) $\Rightarrow$ (c) does not hold if and only if there exists $m$ such that Condition~($*$) is satisfied.
 \begin{description}
  \item[\hspace*{-24pt}Condition ($\bm{*}$) ] $\#S(m) = \frac{(d-2)!}{(d-\ell)!}$, $\#\mathfrak{S}(m) \geq 2$ and the action of $\mathfrak{S}(m)$ on $S(m)$ is trivial.
 \end{description}
 Here, note that if $\# \mathfrak{S}(m) \geq 2$, then there exist $1 \leq i < j \leq \ell$ such that $(i,j) \in \mathfrak{S}(m)$.
 
 In the case $\ell \geq 4$, we have $\sigma \cdot \zeta \ne \zeta$ for every $\sigma = (i,j) \in \mathfrak{S}(m)$ and $\zeta \in S(m)$.
 If $\ell = 2$, we always have $\#\mathfrak{S}(m)=1$.
 Hence in the case $\ell \ne 3$, Condition~($*$) is not satisfied for every $m$.

 We consider the case $\ell = 3$ next.
 In this case, we always have $\#\mathfrak{S}(m) \leq 2$.
 Hence under Condition~($*$), we have $\#\mathfrak{S}(m) = 2$,
 and may assume that $\mathfrak{S}(m)$ consists of the identity and $\sigma := (1,2)$.
 Under this assumption, the action of $\mathfrak{S}(m)$ on $S(m)$ is trivial if and only if
 $\sigma \cdot \zeta = \zeta$ holds for every $\zeta = (\zeta_1:\zeta_2) \in S(m)$.
 However for $\zeta = (\zeta_1:\zeta_2) \in S(m)$, the equality $\sigma \cdot \zeta = \zeta$ holds
 if and only if $(\zeta_2:\zeta_1) = (\zeta_1:\zeta_2)$ holds, 
 which is also equivalent to $\zeta = (1:-1)$.
 Hence $S(m)$ must be equal to $\{ (1:-1) \}$.

 Summing up the above mentioned, Condition~($*$)
 implies $\ell =3$ and $\#S(m) = \frac{(d-2)!}{(d-\ell)!} = 1$, which also implies $d=3$.
 Hence except in the case $d=\ell = 3$, Condition~($*$) is not satisfied for every $m$, 
 and the implication (a) $\Rightarrow$ (c) holds.
 
 To summarize, we have completed the proof of Main Theorem.
\end{proof}

\end{document}